\documentclass[reqno]{amsart}

\usepackage{amsmath,amssymb,amsthm,amsfonts}
\usepackage{color}
\usepackage{hyperref}
\usepackage[active]{srcltx} 
\theoremstyle{plain}
\newtheorem{theorem}{Theorem}[section] 
\newtheorem{lemma}[theorem]{Lemma} 
\newtheorem{prop}[theorem] {Proposition} 
\newtheorem{cor}[theorem]  {Corollary}

\theoremstyle{definition} 
 
\newtheorem{assumption}{Assumption}

\theoremstyle{remark}
\newtheorem*{remark}{Remark} 
\newtheorem*{example}{Example}

\newcommand{\N}{\mathbb{N}}
\newcommand{\R}{\mathbb{R}}
\newcommand{\Z}{\mathbb{Z}}

\renewcommand{\P}{\mathbb{P}}

\newcommand{\dd}{\mathrm{d}} %integration d
\newcommand{\eps}{\epsilon}

\newcommand{\vect}[1]{\boldsymbol{#1}}

\DeclareMathOperator{\Int}{Int}
\DeclareMathOperator{\supp}{supp}
\DeclareMathOperator{\const}{const}
\DeclareMathOperator{\dist}{dist}

\newcommand{\comp}{\mathrm{c}} %complement of a set

\newcommand{\may}{\mathsf{May}}
\newcommand{\cl}{{\rm cl}} 
\newcommand{\per}{\mathrm{per}}
\newcommand{\can}{\mathrm{can}}

\def\1{{\mathchoice {1\mskip-4mu\mathrm l}      % Blackboard bold 1 
{1\mskip-4mu\mathrm l} 
{1\mskip-4.5mu\mathrm l} {1\mskip-5mu\mathrm l}}}

\begin{document} 

%opening
\title[Continuum percolation for Gibbsian point processes]{Continuum percolation for Gibbsian point processes with attractive interactions}

\author{Sabine Jansen} 

\begin{abstract}
	We study the problem of continuum percolation in infinite volume Gibbs measures for particles with an attractive pair potential, with a focus on low temperatures (large $\beta$). The  main results are bounds on percolation thresholds $\rho_\pm(\beta)$ in terms of the density rather than the chemical potential or activity. In addition, we prove a variational formula for a large deviations rate function for cluster size distributions. This formula establishes a link with the Gibbs variational principle and a form of equivalence of ensembles, and allows us to combine knowledge on finite volume, canonical Gibbs measures with infinite volume, grand-canonical Gibbs measures. \\
	
	
\end{abstract}

\maketitle

\section{Introduction} 

The present article is concerned with percolation properties for Gibbsian point processes. We are interested in infinite volume Gibbs measures (in the sense of the Dobrushin-Lanford-Ruelle  conditions) for particles in $\R^d$ interacting via an attractive, finite range pair potential. The dimension is two or higher, $d\geq 2$. Around each particle $x$ of a random configuration $\omega$, draw a ball of radius $R$, for some fixed $R>0$. \emph{Percolation} occurs if the region in $\R^d$ covered by the union of such balls, $\cup_{x\in \omega} \overline{B(x,R)}$, has an unbounded connected component, with positive probability. In the notation of Meester and Roy \cite{meester-roy-book}, our problem is a Boolean percolation model $(X,\rho)$ driven by a Gibbsian point process $X$ and with deterministic radius $\rho = R$. 

This problem has been studied before \cite{muermann75,zessin08,pechersky-yambartsev09,aristoff12}, for both repulsive pair potentials and potentials with an attractive part.  M{\"u}rmann \cite{muermann75} investigated finite-range potentials and gave a sufficient condition for the absence of percolation; a different proof, with an extension to tempered boundary conditions for attractive potentials, was given by Zessin \cite{zessin08}. His proof builds on integration by parts for Gibbsian measures. Pechersky and Yambartsev \cite{pechersky-yambartsev09} proved a criterion for absence of percolation with the help of a coupled branching process; their result does not require that the potential has compact support. In addition, Pechersky and Yambartsev gave a sufficient condition for the presence of percolation, valid in dimension $2$, for attractive potentials with possibly unbounded support. An analogous result, for hard spheres in dimension $2$ was shown by Aristoff \cite{aristoff12}. 

The cited works all formulate criteria in terms of the activity $z$  -- i.e., the intensity parameter of some a priori Poisson point process -- rather than the density $\rho$, which for interacting particles is a non-trivial function of $z$. As a consequence, for attractive potentials, the cited results cannot distinguish between two very different physical pictures. First, percolation might be a high-density effect, as expected for hard spheres; second, it might be an energetic effect -- at low temperature, it may happen that the interaction favors the formation of large connected components,  which because of entropy may coexist with large almost empty regions of space containing a few small components. In such a situation the density threshold for percolation may be very small, as suggested by  recent results \cite{jkm11} on large deviations for cluster size distributions in the canonical ensemble.

The aim of the present article is, therefore, to give bounds on percolation and non-percolation thresholds in terms of the density rather than the activity. Our main result is Theorem \ref{thm:perco-can} below, valid for finite-range, attractive potentials. Quickly summarized, Theorem \ref{thm:perco-can} states that there are curves $\rho_\pm(\beta)$ such that if $P_{\beta,\rho}$ is a shift-invariant Gibbs measure at density $\rho$ and inverse temperature $\beta$, the following holds. 
\begin{itemize}
	\item If $\rho<\rho_-(\beta)$, there are  only bounded connected components, $P_{\beta,\rho}$-almost surely; we have $\rho_-(\beta) = \exp( - \beta \nu^*(1+o(1)))$ for suitable $\nu^*>0$. 
	\item If $\rho>\rho_+(\beta)$, there is an unbounded connected component, $P_{\beta,\rho}$-almost surely; as $\beta \to \infty$, $\rho_+(\beta) \to \rho_0>0$ for suitable $\rho_0>0$.  
\end{itemize}
%As $\beta \to \infty$, the curves satisfy $\rho_-(\beta) = \exp( - \beta \nu^*(1+o(1)))$ and 
%$\liminf_{\beta \to \infty} \rho_+(\beta) \geq \rho_0$ for some $\nu^*>0$ and $\rho_0>0$.  
We expect that for very large $\beta$ and $\rho_-(\beta)<\rho <\rho_+(\beta)$, there are non-ergodic Gibbs measures with percolation probability strictly between $0$ and $1$.
We have no proof of this conjecture; see, however, Proposition \ref{prop:soft-trans} for some preliminary evidence, and Appendix \ref{app:lattice} for relevant lattice gas results.

The key technical tool for the proof of Proposition \ref{prop:soft-trans} is a variational formula for a large deviations rate for cluster size distributions in the canonical ensemble (Theorem \ref{thm:equiv}); the large deviations principle was investigated in \cite{jkm11}. The variational formula establishes a relation with the Gibbs variational principle and allows us to apply a form of equivalence of ensemble \cite{georgii95}. This in turn enables us to combine the knowledge for grand-canonical, infinite volume Gibbs measures in \cite{muermann75,pechersky-yambartsev09} with results on the canonical ensemble \cite{jkm11}. 

%In addition, we show that the proof of percolation in \cite{pechersky-yambartsev09} extends to all dimensions $d\geq 2$, and exploit results on the density-activity function at low temperature given in \cite{jansen12}. 

We conclude the introduction with a word of caution on the physical interpretation of percolation. We should stress that a percolation transition need not be a phase transition -- for the ideal gas at activity $z$ (Poisson point process with intensity $z$), the pressure is an analytic function no matter the value of $z$, even though there is a percolation transition at high enough $z$ \cite{meester-roy-book}. Nevertheless, for attractive pair potentials and at low temperature, the percolation transition might coincide with a phase transition. In fact, for nearest neighbor attractive lattice gases (or Ising model), and temperatures below some threshold $T_+$,  the percolation transition and the phase transition coincide (see the review by Georgii, H{\"a}ggstr{\"o}m and Maes \cite{ghm}). In dimension two, $T_+$ equals the Curie temperature $T_\mathrm{C}$, but in higher dimensions $T_+<T_\mathrm{C}$ \cite{aizenman-bricmont-lebowitz87}, illustrating again that percolation and phase transition in general do not coincide. For the reader's convenience, we summarize some relevant results in Appendix \ref{app:lattice}.

The remainder of the article is organized as follows. In Sections \ref{sec:setting} and  \ref{sec:results} we formulate the setting and our results. Section \ref{sec:topo} defines the topology of local convergence and summarizes continuity properties of important functions such as the relative entropy rate. Sections \ref{sec:var} to \ref{sec:perco} are devoted to the proofs. 

\newpage

\section{Setting} \label{sec:setting} 
  
\subsection{Pair potential}   
  
The pair potential is a function $v:[0,\infty) \to \R \cup \{\infty\}$ that serves to define the total energy of an $N$-particle configuration
\begin{equation*} 
	U_N(x_1,\ldots,x_N):= \sum_{1 \leq i<j\leq N} v(|x_i-x_j|),\quad x_1,\ldots,x_N\in \R^d.
\end{equation*} 
By a slight abuse of notation, we shall drop the subscript and write $U(\vect{x})$ instead of $U_N(\vect{x})$.
  
\begin{assumption} \label{ass:basic}
	The pair potential satisfies the following basic assumptions:
	\begin{itemize} 
	\item Either $v$ is everywhere finite or there is a $r_\mathrm{hc}>0$ such that 
		$v(r)=\infty$ for $r<r_\mathrm{hc}$ and $v(r)<\infty$ for $r>r_\mathrm{hc}$. (We impose no condition on $v(r_\mathrm{hc})$.)
	\item $v$ has compact support: $r_1:= \sup \supp v<\infty$. 
	\item $v$ is bounded from below: $\inf v>-\infty$. 
	\item $v$ has an attractive tail: for suitable $r_0<r_1$ and all $r\in (r_0,r_1)$, $v(r)<0$. 
	\end{itemize}
\end{assumption}
If $r_\mathrm{hc}>0$, we say that $v$ has a \emph{hard core}. If $v$ has compact support, we shall also say that $v$ has \emph{finite range}. 
%
%\begin{assumption} \label{ass:nad}
%	The pair potential is \emph{non-integrably divergent at the origin}: there exists a decreasing function $\chi:(0,\infty)\to [0,\infty)$ such that $v(r) \geq \chi(r)$ for sufficiently small $r$, and $\lim_{r\to 0} \chi(r) =\infty$.  
%\end{assumption}
%Assumptions \ref{ass:basic} and \ref{ass:nad} imply that $v$ is \emph{superstable}. 
%This means the following.  
For $\vect{k} \in \Z^d$, let $C(\vect{k})$ be the unit cube $[k_1+1)\times \cdots \times [k_d,k_d+1)$ and 
\begin{equation*}	
	N_{C(\vect{k})} (x_1,\ldots,x_N) := \bigl| \bigl\{ j \in \{1,\ldots,N\} \big|\, x_j \in C(\vect{k}) \bigr\} \bigr|
\end{equation*}  
the number of particles in  $C(\vect{k})$. 

\begin{assumption} \label{ass:superstable}
	The pair potential $v$ is \emph{superstable}: there are constants $a>0$, $b<\infty$ such 		that for all $N\in \N$ and all $\vect{x}=(x_1,\ldots,x_N )\in (\R^d)^N$,  
	\begin{equation} \label{eq:superstable}
	 		U(\vect{x}) \geq \sum_{\vect{k} \in \Z^d} \Bigl( a N_{C(\vect{k})} (\vect{x}) ^2 - b N_{C(\vect{k})}(\vect{x}) \Bigr). 
\end{equation}
\end{assumption}
Every superstable interaction is \emph{stable}, since Eq. \eqref{eq:superstable} implies in particular that $U(x_1,\ldots,x_N) \geq - b N$. If $v$ satisfies Assumption \ref{ass:basic} and is non-integrably divergent at the origin, then $v$ is superstable; see \cite{ruelle70,ruelle-book} for a proof and other sufficient conditions. 

\begin{assumption} \label{ass:integrable} 
	The potential is integrable in $\{v<\infty\}$: 
%	\begin{equation*} 
	 $ \int_{|x|>r_\mathrm{hc}} |v\bigl(|x|\bigr)| \dd x <\infty.$ 
%	\end{equation*}
\end{assumption}

\begin{assumption}  \label{ass:interparticle-lower}
	 There is an $r_\mathrm{min}>0$ such that for every $N\in \N$, 
	$U$ has a minimizer $(x_1,\ldots,x_N) \in (\R^d)^N$ with interparticle distance bounded from below by $r_\mathrm{min}$: 
	\begin{equation*} 
		1 \leq i<j \leq N \Rightarrow |x_i -x_j| \geq r_\mathrm{min}. 
	\end{equation*}
	In addition, $v$ is H{\"o}lder-continuous in $[r_\mathrm{min},\infty)$. 
\end{assumption}
A sufficient condition for the lower bound on interparticle distances is that $v(r)/r^d \to \infty$ as $r\to 0$, as can be shown along \cite[Lemma 2.2]{theil06}.

%\begin{assumption}
%	Let $r_0$ be as in Assumption \ref{ass:interparticle-lower}. The pair potential $v$ is continuous in $[r_0,r_0+\eps)$, for suitable $\eps>0$.  
%\end{assumption}
\begin{assumption} \label{ass:interparticle-upper}
	There is a $C>0$ such that for every $N\in\N$, $U$ has a minimizer $(x_1,\ldots,x_N) \in (\R^d)^N$ with diameter bounded by $C N^{1/d}$: 
		\begin{equation*} 
			\max \{ |x_i - x_j| \mid 1\leq i <j \leq N \} \leq C N^{1/d}. 
		\end{equation*} 
\end{assumption}
Under Assumption \ref{ass:basic}, this condition is trivially fulfilled in dimension $1$. In dimension $2$, sufficient conditions are given, for example, in \cite{theil06}, where much more is proven on the ground states. To the best of our knowledge, there is no result in dimension $3$ or higher; in fact, providing upper bounds on interparticle distances for Lennard-Jones type interactions seems to be a non-trivial problem in non-linear optimization, see the article by Blanc \cite{blanc04} and the references therein. 

Let us briefly comment on our conditions on the pair potential. Assumption \ref{ass:basic} simply defines the class of pair potentials we are interested in. Superstability as in Assumption \ref{ass:superstable} is a standard condition that ensures the existence of infinite volume Gibbs measures, see the next subsection. The integrability assumption \ref{ass:integrable} will allow us to use a bound on Mayer expansions going back to Brydes and Federbush \cite{brydges-federbush78}.
% the additional assumption \ref{ass:nad} of non-integrable divergence at the origin enters the proof of the Gibbs variational principle in \cite{georgii94}. 
Assumptions \ref{ass:interparticle-lower} and \ref{ass:interparticle-upper} are needed for precise statements about Gibbs measures at low temperature $\beta^{-1}$ and densities above some threshold of the form $\exp(- \beta \nu^*)$. 

\subsection{Infinite volume Gibbs measures}
Let $\Omega$ be the set of locally finite point configurations, 
\begin{equation*}
	\Omega:= \{ \omega \subset \R^d \mid \forall r>0:\ |\omega\cap B(0,r)|<\infty \},  
\end{equation*} 
with $B(0,r)$ the open ball of radius $r$ centered at the origin. We equip
$\Omega$  with  the $\sigma$-algebra $\mathcal{F}$  generated by the counting variables $N_B(\omega):=|\omega\cap B|$, $B\subset \R^d$ Borel-measurable, and denote the probability measures on $(\Omega,\mathcal{F})$ with the letter $\mathcal{P}$. 
The following subsets of $\mathcal{P}$ will be relevant for us: shift-invariant measures $\mathcal{P}_\theta$, tempered measures, and infinite-volume Gibbs measures $\mathcal{G}(\beta,\mu)$; we proceed with their definition. 

 For $x\in \R^d$, the shift $\theta_x:\Omega \to \Omega$ is defined by $\theta_x\omega:= \{ y-x \mid y \in \omega \}$. A measure $P\in \mathcal{P}$ is \emph{shift-invariant} if $P(\theta_x(A))=P(A)$ for all measurable $A\subset \Omega$ and all $x\in \R^d$. The collection of shift-invariant measures is denoted $\mathcal{P}_\theta$. 
We say that $P \in \mathcal{P}$ is \emph{tempered} if for $P$-almost all $\omega$, 
\begin{equation} \label{eq:tempered}
	\exists t(\omega)>0\ \forall \ell \in \N: \quad \sum_{\vect{k} \in \Z^d \cap [-\ell,\ell]^d} \Bigl( N_{C(\vect{k})}(\omega)\Bigr) ^2 \leq t(\omega) \ell^d, 
\end{equation} 
where $C(\vect{k})$ is the unit cube $[k_1,k_1+1)\times \cdots \times [k_d,k_d+1)$. 
Fix $\beta>0$ and $\mu \in \R$. Let $\Lambda \subset \R^d$ be a bounded Borel set and 
$\zeta \in \Omega$ a configuration satisfying the temperedness condition \eqref{eq:tempered}. For $n \in \N$, define the measure $Q_n$ on $\Lambda^n$ as the measure with Lebesgue density 
\begin{equation*} 
%	\frac{1}{\Xi_{\Lambda \mid \omega \cap \Lambda^\comp} (\beta,\mu)} 
	 \frac{1}{n!} \exp \Bigl( - \beta \sum_{1\leq i<j \leq n} v(|x_i-x_j|) 
										- \beta \sum_{i=1}^n \sum_{y \in \zeta \cap\Lambda^\comp} 
												v(|x_i - y|) \Bigr).
\end{equation*}
The image of $Q_n$ under $(x_1,\ldots,x_n) \mapsto \{x_1,\ldots,x_n\} $ is a measure on $\Omega$ which we denote again by $Q_n$. Let $Q_0$ be the probability measure on $\Omega$ which gives probability $1$ to the event that $\omega = \emptyset$. Define $P_{\beta,\mu,\Lambda \mid \zeta}$ by
\begin{equation*}
	P_{\beta,\mu,\Lambda \mid \zeta}:= \frac{1}{\Xi_{\Lambda \mid \zeta}(\beta,\mu)} 
			\sum_{n=0}^\infty z^nQ_n, \quad z = \exp(\beta \mu). 
\end{equation*} 
The normalization $\Xi_{\Lambda \mid \zeta}(\beta,\mu)$ is defined by the requirement that $P_{\beta,\mu,\Lambda \mid \zeta}$ is a probability measure on $\Omega$. 
 
 We say that $P\in \mathcal{P}$ is a \emph{$(\beta,\mu)$-Gibbs measure} if and only if $P$ is tempered and for every $\Lambda$ and every measurable $f:\Omega \to [0,\infty)$, 
 \begin{equation} \label{eq:dlr} 
 	\int_\Omega P(\dd \omega)  f(\omega) = \int_\Omega P(\dd \zeta) \int_\Omega P_{\beta,\mu,\Lambda \mid \zeta} (\dd \omega) f\bigl(\omega \cup (\zeta \cap \Lambda^\comp) \bigr). 
 \end{equation}
We denote the set of $(\beta,\mu)$-Gibbs measures as $\mathcal{G}(\beta,\mu)$; the sets $\mathcal{G}(\beta,\mu)$, $(\beta,\mu) \in \R_+\times \R$, are pairwise disjoint \cite[Remark 3.7]{georgii95}. The existence of Gibbs measures for lower regular, superstable potentials -- this class includes our potentials satisfying Assumptions \ref{ass:basic} and \ref{ass:superstable} --  is shown in \cite{ruelle70,dobrushin70}: for all $\beta>0$, $\mu \in \R$, $\mathcal{G}(\beta,\mu)$  and $\mathcal{G}(\beta,\mu)\cap \mathcal{P}_\theta$ are non-empty, convex sets. With suitable topologies, they are Choquet simplices and, in particular, compact  \cite[Theorems 5.6 and 5.8]{ruelle70}.

\subsection{Palm measure; energy and entropy densities}
%For a shift-invariant measure $P \in \mathcal{P}_\theta$, let $\rho(P):=\int_\Omega %N_{[0,1]^d}(\omega) P(\dd \omega)$ be the expected number of particles per unit volume. 
For each $P\in \mathcal{P}_\theta$, there is a unique finite measure $P^\circ$ on $(\Omega,\mathcal{F})$ such that for all measurable, non-negative functions $f:\R^d\times \Omega \to [0,\infty)$,  
\begin{equation} \label{eq:palm}
	 \int_\Omega \sum_{x \in \omega} f(x,\theta_x \omega) P(\dd \omega) = \int_{\Omega} \int_{\R^d}  f(x,\omega) \dd x P^\circ (\dd \omega). 
\end{equation}
$P^\circ$ is the \emph{Palm measure} of $P$ \cite[Chapter 13]{daley-vere-jones-vol2}. 
It is known that $P^\circ(0\notin \omega)=0$ and $P^\circ(\Omega)=\rho(P)$, where $\rho(P)$ is the expected number of particles in $[0,1)^d$, called \emph{density}.
The defining equation \eqref{eq:palm}, specialized to $f(x,\omega)= \mathbf{1}_C(x) g(\omega)$ with $C\subset \R^d$ Borel-measurable, yields the useful identity
\begin{equation} \label{eq:gpalm}
	\int_\Omega g(\omega) P^\circ(\dd \omega) = \frac{1}{|C|} \int_\Omega \sum_{x\in \omega \cap C} g(\theta_x \omega) P(\dd \omega). 
\end{equation} 
$|C|$ is the Lebesgue volume of $C$.

\begin{remark} The Palm measure is the continuum analogue of a simple lattice object: if $\Omega=\{ \omega \mid \omega \subset \Z^d\} \equiv \{0,1\}^{\Z^d}$ and $P$ is a shift-invariant measure on $\Omega$, the Palm measure becomes 
$P^\circ(A) = P(A \cap \{0 \in \omega\})$. 
\end{remark} 

For a shift-invariant, tempered measure $P$, we define  the expected energy per unit volume, or \emph{energy density}, as
\begin{equation*} 
	U(P):= \frac{1}{2}\int_\Omega \Bigl( \sum_{y\in \omega} v(|y|)\Bigr) P^\circ (\dd \omega) 
		=  \frac{1}{2} \int_\Omega \Bigl( \sum_{x\in \omega\cap [0,1]^d} \sum_{y \in \omega} v(|y-x|)\Bigr) P(\dd \omega).
\end{equation*}
$U(P)$ takes values in $\R\cup \{\infty\}$. For superstable potentials as in Eq.  \eqref{eq:superstable}, the energy is bounded from below: $U(P) \geq - b^2/4a$. 
%$$U(P)\geq a \rho(P)^2- b\rho(P) \geq - b^2/4a.$$

The entropy per unit volume, or \emph{entropy density} is defined as a relative entropy rate. Let $Q \in \mathcal{P}_\theta$ be the Poisson point process with intensity $1$ and $P\in \mathcal{P}_\theta$. For $\Lambda =[-L,L]^d\subset \R^d$, let $P_\Lambda$ be the image of $P$ under the projection map $\omega \mapsto \omega \cap \Lambda$. Define $Q_\Lambda$ in a similar way. The relative entropy is 
\begin{equation*} 
	I(P_\Lambda;Q_\Lambda) := \int_\Omega f(\omega) \log f(\omega) Q_\Lambda(\dd \omega) = \int_\Omega \log f(\omega) P_\Lambda(\dd \omega)
\end{equation*} 
if $P_\Lambda$ has Radon-Nikodym derivative $\dd P_\Lambda / \dd Q_\Lambda=f$, and $I(P_\Lambda;Q_\Lambda):= \infty$ if $P_\Lambda$ s not absolutely continuous with respect to  $Q_\Lambda$.
%\begin{equation*} 
%	I(P_\Lambda;Q_\Lambda) = 
%		\begin{cases}
%				\int_\Omega f(\omega) \log f(\omega) Q_\Lambda(\dd \omega),&\quad 
%						\text{if }P_\Lambda \ll Q_\Lambda\ \text{with Radon-Nikodym derivative } f,\\
%					\infty,\quad &\quad \text{if }P_\Lambda\ \text{is not absolutely continuous w.r.t. } Q_\Lambda
%		\end{cases}
%\end{equation*} 
The limit 
\begin{equation}  \label{eq:entropy}
	S(P):= 1 -\lim_{|\Lambda| \to \infty} \frac{1}{|\Lambda|} I(P_\Lambda;Q_\Lambda) \in \R \cup \{-\infty\}
\end{equation}
exists for all $P\in \mathcal{P}_\theta$, see \cite{georgii-zessin93}. Continuity properties of $U(P)$ and $S(P)$ with respect to a suitable topology on $\mathcal{P}$ are recalled in Section \ref{sec:topo} below.

\begin{remark}
	The additive constant $1$ is included in Eq. \ref{eq:entropy} for aesthetic reasons; if we did not include it, we would need an additive constant in Eq. \eqref{eq:var}. 
\end{remark}

\begin{example} Let $P$ be a Poisson point process with intensity parameter $z$. Then
\begin{align*} 
	I(P_\Lambda;Q_\Lambda) & = \sum_{n=0}^\infty \Bigl(
		 \log \frac{ (z|\Lambda|)^n\exp(- z|\Lambda|)}{|\Lambda|^n \exp(-|\Lambda|)} \Bigr) \frac{(z|\Lambda|)^n}{n!} e^{-z|\Lambda|} \\
		& = (1-z) |\Lambda| + z|\Lambda| \log z 
\end{align*}
thus $S(P) = - z(\log z -1)$.  
\end{example}

\subsection{Cluster densities} 
Let $r_1>0$ be the range of the potential as in Assumption \ref{ass:basic}. Fix $R\geq r_1$. 
For $\omega \in \Omega$, let $G_\omega$ be the graph with vertex set $\omega$ and edge set $\{ \{x,y\} \mid x,y\in \omega,\ 0<|x-y|\leq R\}$. When $x\in \omega$, let  $\mathcal{C}_\omega(x)\subset \omega$ be the connected component of $x$ in $G_\omega$. 
%We say that $x$ is in a $k$-cluster if $\mathcal{C}_\omega(x)$ has cardinality $k$. 
When $x\in \R^d \backslash \omega$, we set $\mathcal{C}_\omega(x):=\emptyset$ and $|\mathcal{C}_\omega(x)|:=0$. For $P\in \mathcal{P}_\theta$ and $k\in \N$,  the \emph{expected number of $k$-clusters per unit volume} or \emph{$k$-cluster density} is \begin{equation}\label{eq:rhokp} 
	\rho_k(P):= k^{-1}P^\circ\bigl( |\mathcal{C}_\omega(0)|=k\bigr) 
		= \frac{1}{|C|} \int_\Omega \sum_{x\in \omega \cap C} \mathbf{1}\bigl(  |\mathcal{C}_\omega(x)|=k\bigr) P(\dd \omega). 
\end{equation}
The last identity holds for every Borel set $C$, compare Eq. \eqref{eq:gpalm}. 
Note that for every $P\in \mathcal{P}_\theta$, $\sum_{k=1}^\infty k \rho_k(P)\leq \rho(P)$.
For later purpose we also define finite volume empirical densities as  
\begin{equation*} 
	\rho_{k,\Lambda}(\omega):= \frac{1}{k|\Lambda|} \sum_{x\in \omega} \mathbf{1}\bigl( |\mathcal{C}_{\omega \cap \Lambda}(x)| =k\bigr).
\end{equation*}
Thus $\rho_{k,\Lambda}(\omega)$ is the number of $k$-clusters in $\omega \cap \Lambda$, divided by the volume $|\Lambda|$. We have, for all $\omega \in \Omega$, 
$\sum_{k=1}^{N_\Lambda(\omega)} k\rho_{k,\Lambda}(\omega) = N_\Lambda(\omega)/|\Lambda|$.

\subsection{Large deviations for cluster size distributions} 
Finally we recall a large deviation principle shown in \cite{jkm11}. Fix $R>r_1$. 
For $\Lambda=[0,L]^d$ and $N\in \N$, let 
\begin{equation*} 
	Z_\Lambda(\beta,N):=\frac{1}{N!} \int_{\Lambda^N} e^{-\beta U(x_1,\ldots,x_N)} \dd x_1\cdots \dd x_N.
\end{equation*} 
be the canonical partition function and 
$\P_{\beta,N,\Lambda}$  the probability measure on $\Lambda^N$ 
with density $N!^{-1} Z_\Lambda(\beta,N)^{-1} \exp(- \beta U(x_1,\ldots, x_N))$. 
The image of $\P_{\beta,N,\Lambda}$ under the mapping $\R^N\to \Omega$, $(x_1,\ldots,x_N) \mapsto \{x_1,\ldots,x_N\}$ is a measure on $(\Omega,\mathcal{F})$, for which by a slight abuse of notation we use the same letter $\P_{\beta,\Lambda,N}$. 

Equip $\R_+^\N$ with the product topology and the associated Borel $\sigma$-algebra, and let $\vect{\rho}_\Lambda: (\Omega,\mathcal{F},\P_{\beta,N,\Lambda}) \to \R_+^\N$ be the random variable $\vect{\rho}_\Lambda(\omega):= (\rho_{k,\Lambda}(\omega))_{k\in \N}$. 
We are interested in the behavior of $\vect{\rho}_\Lambda$ in the \emph{thermodynamic limit } 
\begin{equation} \label{eq:thermolim} 
	N\to \infty, \quad L\to \infty,\quad \frac{N}{|\Lambda|}\to \rho 
\end{equation}
for $\rho>0$. First we recall the definition of the free energy $f(\beta,\rho)$ and the close-packing density $\rho_\mathrm{cp}$.  Consider the limit 
\begin{equation*} 
	f(\beta,\rho):=- \lim \frac{1}{\beta |\Lambda|} \log Z_\Lambda(\beta,N)
\end{equation*} 
along \eqref{eq:thermolim}. 
It is well-known \cite{ruelle-book} that the limit exists and is finite when the density $\rho$ is strictly smaller than the \emph{close-packing density} $\rho_\mathrm{cp}>0$, and is infinite when $\rho>\rho_\mathrm{cp}$. When the potential has no hard core ($r_\mathrm{hc}=0$), we have $\rho_\mathrm{cp}=\infty$ and the limit $f(\beta,\rho)$ is finite for all $\rho>0$. 
 
The following holds \cite{jkm11}:  for all $\beta>0$ and $\rho\in (0,\rho_\mathrm{cp})$, in the limit \eqref{eq:thermolim}, the random variable  $\vect{\rho}_\Lambda(\omega)$  satisfies a large deviation principle with speed $\beta |\Lambda|$. The rate function is of the form 
$f(\beta,\rho,(\rho_k)_{k\in \N}) - f(\beta,\rho)$ for a  function $f(\beta,\rho,\cdot): [0,\infty)^\N\to \R\cup \{\infty\}$ that is convex and lower semi-continuous with compact sublevel sets.
% If $f(\beta,\rho,(\rho_k)_k) <\infty$, then necessarily $\sum_{k=1}^\infty k \rho_k \leq \rho$.
Moreover,
\begin{equation*} 
	f(\beta,\rho) = \min\Bigl\lbrace f(\beta,\rho,(\rho_k)_{k\in\N}\bigr) \,\big|\, (\rho_k)_{k\in \N} \in [0,\infty)^\N,\ \sum_{k=1}^\infty  k \rho_k \leq \rho \Bigr\rbrace, 
\end{equation*} 
%
%and every limit law of $(\rho_{k,\Lambda}(\omega))_{k\in \N}$ must be supported on the set of minimizers of $f(\beta,\rho,\cdot)$. In particular, when $f(\beta,\rho,\cdot)$ has a unique minimizer $(\rho_k)_{k\in \N}$, we know that $(\rho_{k,\Lambda}(\omega))_{k\in \N}$ converges in law to $(\rho_k)_{k\in \N}$. 
Note that the rate function $f(\beta,\rho,\cdot)$ (but not the free energy $f(\beta,\rho)$!) depends on the connectivity radius  $R\geq r_1$; to lighten notation, however, we leave the $R$-dependence implicit.

\section{Results} \label{sec:results}

Here we formulate our main results. Section \ref{sec:var-char} provides a variational characterization of percolation. Sections \ref{sec:perco-gc} and \ref{sec:perco-can} formulate bounds on percolation thresholds in terms of the chemical potential $\mu = \beta^{-1} \log z$ and the density $\rho$. Section \ref{sec:var-char} on the one hand and Sections \ref{sec:perco-gc} and \ref{sec:perco-can} on the other hand are logically independent, except for Proposition \ref{prop:soft-trans}. 

Throughout the remainder of the article we shall assume without further mention that the pair potential satisfies Assumptions \ref{ass:basic} and \ref{ass:superstable}, and the connectivity radius $R$ is larger or equal to the potential range $r_1$. 

\subsection{Variational characterization of percolation}  \label{sec:var-char}
Our first result expresses the rate function $f(\beta,\rho,(\rho_k))$ of \cite{jkm11} in terms of a variational problem. 
%We adopt the convention that $\inf \emptyset = \infty$ and recall that for sufficiently low density, $\sum_1^\infty k \rho_k \leq \rho$ is enough to ensure that $f(\beta,\rho,(\rho_k))<\infty$. 

\begin{theorem} %[Variational formula for $f(\beta,\rho,\{\rho_k\})$]
 \label{thm:var}
 	For every $\rho \in (0,\rho_\mathrm{cp})$, $\beta >0$ and $(\rho_k)_{k\in \N} \in \R_+^\N$, 
	\begin{multline} \label{eq:var}
		f\bigl(\beta,\rho, (\rho_k)_{k\in \N}\bigr)\\
				 = \min\bigl \lbrace U(P) -  \beta^{-1} S(P) \mid 
		P\in \mathcal{P}_\theta,\ \rho(P) = \rho,\ 
			\forall k \in \N:\  \rho_k(P) = \rho_k \bigr \rbrace
	\end{multline}
	with the convention $\min \emptyset =\infty$. 
\end{theorem}

We recall Theorem 3.4  from~\cite{georgii95}. Fix $\beta>0$ and  $\mu \in \R$. Then $P \in \mathcal{P}_\theta$ is a $(\beta,\mu)$-Gibbs measure if and only if 
it minimizes $U(P) - \beta^{-1} S(P) - \mu \rho(P)$; the minimum $-p(\beta,\mu)$ is minus the  \emph{pressure}. This is the  \emph{Gibbs variational principle}. Together with Theorem \ref{thm:var}, it allows us to establish the following relation between minimizers of $f(\beta,\rho,\cdot)$ and infinite volume Gibbs measures.\footnote{Georgii states his result under the additional assumption that the potential is non-integrably divergent at the origin. A close inspection of the proof shows, however, that we can dispense with this condition because our potentials have finite range; see the comment in \cite{georgii95} before Lemma 7.3.}

\begin{theorem}\label{thm:equiv}
	Fix $\rho \in (0,\rho_\mathrm{cp})$ and $\beta>0$. Then, for every $(\rho_k)_{k\in \N} \in \R_+^\N$,  the following two statements are equivalent: 
	\begin{enumerate}
		\item $(\rho_k)_{k\in \N}$ is a minimizer of $f(\beta,\rho, \cdot)$. 
		\item There is a chemical potential $\mu \in \R$ and a shift-invariant Gibbs measure $P\in \mathcal{P}_\theta \cap \mathcal{G}(\beta,\mu)$ such that $\rho(P)=\rho$ and for all $k \in \N$, $\rho_k(P) = \rho_k$. 
	\end{enumerate}	
\end{theorem}
The next elementary proposition establishes a relation between cluster densities and percolation, valid for general shift-invariant point processes. 
\begin{prop} \label{prop:perco}
	Let $P \in \mathcal{P}_\theta$. The following statements are equivalent: 
	\begin{enumerate}
		\item $\sum_{k=1}^\infty k \rho_k(P) <\rho(P)$. 
		\item $P(\text{there is a cluster with infinitely many particles})>0$. 
		\item $P(\text{there is a cluster with infinite diameter}) >0$. 
	\end{enumerate}
\end{prop}

A quantitative relation is given in Eq. \eqref{eq:deficit-perco} below. 
We shall refer to both (2) and (3) as $P(\text{there is an infinite cluster})>0$, or more briefly as percolation.  
An immediate consequence of Theorem \ref{thm:equiv} and Prop. \ref{prop:perco}  is the following characterization of percolation in shift-invariant Gibbs measures. Set
\begin{equation*}
	\mathcal{G}_\theta(\beta,\rho):=  \Bigl\{ P\in\bigcup_{\mu \in \R} \mathcal{G}(\beta,\mu) \mid P \in \mathcal{P}_\theta,\ \rho(P)=\rho  \Bigr\}.
\end{equation*}
The Gibbs variational principle implies that  $\mathcal{G}_\theta(\beta,\rho)$ 
consists of the minimizers of $U(P)-\beta^{-1} S(P)$ under the constraint $\rho(P)=\rho$; compare with the proof of Theorem \ref{thm:equiv} below.

\begin{cor} \label{cor:char} 
	Let $\beta >0$ and $\rho \in (0,\rho_\mathrm{cp})$. Consider the following statements:
		\begin{enumerate}
			\item There is a $P \in  \mathcal{G}_\theta(\beta,\rho)$ such that $P(\text{there is an infinite cluster}) >0$. 
			\item $f(\beta,\rho,\cdot)$ has a  minimizer $(\rho_k)$ such that $\sum_{k=1}^\infty k \rho_k =\rho$. 
			\item For every $P \in  \mathcal{G}_\theta(\beta,\rho)$, $P(\text{there is an infinite cluster}) =1$. 
			\item Every minimizer $(\rho_k)$ of $f(\beta,\rho,\cdot)$ satisfies $\sum_{k=1}^\infty k \rho_k <\rho$. 
%			\item[(1'')] For every $P \in  \mathcal{G}_\theta(\beta,\rho)$, $P(\text{there is an infinite cluster}) =0$. 
%			\item[(2'')] Every minimizer $(\rho_k)$ of $f(\beta,\rho,\cdot)$ satisfies $\sum_{k=1}^\infty k \rho_k =\rho$. 
		\end{enumerate}
		We have (1)$\Leftrightarrow$(2) and (3)$\Leftrightarrow$(4).  
\end{cor}			
A similar characterization holds for non-percolation.
We note that Corollary \ref{cor:char} establishes a relation between percolation for infinite volume Gibbs measures and cluster size distributions in finite volume canonical ensembles as examined in \cite{jkm11}. 
 
\subsection{Percolation thresholds: grand-canonical ensemble} \label{sec:perco-gc}

Set $E_1:=0$ and
\begin{equation*}
	E_k := \inf_{(\R^d)^k} U(x_1,\ldots,x_k) \ (k \in \N),\quad e_\infty:= \inf_{k\in\N} \frac{E_k}{k}>-\infty.
\end{equation*}
Let $r_0<r_1$ as in Assumption \ref{ass:basic} and suppose that $v$ is continuous in $(r_0,r_1)$. Set $-M:= \inf_{r>r_0} v(r) <0.$ For $m<M$, choose $\tilde r_m \in (r_0,r_1)$ such that $v(\tilde r_m) \leq -m$, 
%For $m \in (0,M)$, let $\tilde r_m>r_0$ such that 
%%\begin{equation*}
%	$\inf_{r>\tilde r_m} v(r) \leq - m,$
%%\end{equation*}
see Figure 2 in \cite{pechersky-yambartsev09}. Note $e_\infty \leq - d m$. 
Consider the conditions 
\begin{align} 
	&\forall \mu >\mu_+\  \forall P\in \mathcal{G}(\beta,\mu)\ 
		  P(\text{there is an infinite $R$-cluster}) =1  \label{eq:cond-mu-plus} \\
	& \forall \mu<\mu_-\  \forall P\in \mathcal{G}(\beta,\mu)\ P(\text{there is an infinite $R$-cluster}) =0.\label{eq:cond-mu-minus} 
\end{align}
Set 
\begin{align*} 
	\mu_+(\beta;R) &:= \inf \{ \mu_+\in\R \mid \mu_+ \text{ satisfies \eqref{eq:cond-mu-plus}} \},  \\
	\mu_-(\beta;R) &:= \sup \{ \mu_-\in\R \mid \mu_- \text{ satisfies \eqref{eq:cond-mu-minus}} \}. 
\end{align*} 
Clearly $\mu_-(\beta;R) \leq \mu_+(\beta;R)$. 

\begin{theorem} \label{thm:perco-gc} 
	Let $R\geq r_1$, $m \in (0,M)$ and $R_m > \sqrt{d+3}\, \tilde r_m \geq \sqrt{d+3} r_0$. 
\begin{enumerate}
	\item %Let $R\geq r_1$. 
		Suppose that $v$ satisfies Assumption \ref{ass:integrable}. Then 
		%\begin{equation} 
			$$e_\infty \leq \liminf_{\beta\to \infty} \mu_-(\beta;R).$$
		%\end{equation} 
		In addition, for every $\mu <e_\infty$ and sufficiently large $\beta$, there is a unique $(\beta,\mu)$-Gibbs measure $P$; it is shift-invariant, has no infinite cluster ($P$-almost surely),  and satisfies 
		\begin{equation} \label{eq:decay-gc}
			P^\circ( |\mathcal{C}_\omega(0)|=k \bigr) = k \rho_k(P) \leq k e^k |B(0,R)|^{k-1} 
				\exp\bigl(- \beta k (e_\infty-\mu) \bigr).
		\end{equation} 
	\item %Let $R_m>\max(\sqrt{d+3}\, \tilde r_m, r_1)$. 
		Suppose that $v$ is continuous in $(r_0,r_1)$. Then 
	\begin{equation*}
		\limsup_{\beta \to \infty} \mu_+(\beta;R_m) \leq - m.
	\end{equation*}
\end{enumerate}
\end{theorem}

We conjecture that for every fixed $R\geq r_1$, under suitable conditions on $v$, we 
 have $\mu_-(\beta;R)=\mu_+(\beta;R)$ for sufficiently large $\beta$ and 
	\begin{equation*} 
		\lim_{\beta \to \infty} \mu_-(\beta;R)= \lim_{\beta \to \infty} \mu_+(\beta;R) = e_\infty. 
	\end{equation*}
See Appendix \ref{app:lattice} for the corresponding lattice gas result.

\subsection{Percolation thresholds: canonical ensemble} \label{sec:perco-can} 
Let 
$\nu^*:= \inf_{k\in \N} (E_k - k e_\infty)$. It is known that for potentials with an attractive tail, $\nu^*>0$ \cite{jkm11}. 
Consider the conditions 
\begin{align} 
	&\forall \rho >\rho_+\  \forall P\in \mathcal{G}_\theta(\beta,\rho)\ 
		  P(\text{there is an infinite $R$-cluster}) =1,  \label{eq:cond-rho-plus} \\
	& \forall \rho<\rho_-\  \forall P\in \mathcal{G}_\theta(\beta,\rho)\ P(\text{there is an infinite $R$-cluster}) =0.\label{eq:cond-rho-minus} 
\end{align}
Set 
\begin{align*} 
	\rho_+(\beta;R) &:= \inf \{ \rho_+\in (0, \rho_\mathrm{cp}) \mid \rho_+ \text{ satisfies \eqref{eq:cond-rho-plus}} \},  \\
	\rho_-(\beta;R) &:= \sup \{ \rho_-\in (0, \rho_\mathrm{cp}) \mid \rho_- \text{ satisfies \eqref{eq:cond-rho-minus}} \}. 
\end{align*} 
Clearly $\rho_-(\beta;R) \leq \rho_+(\beta;R)$. For $\mu \in \R$, we define $\rho(\beta,\mu)$ as the smallest density $\rho$ of Gibbs measures $P \in \mathcal{P}_\theta \cap \mathcal{G}(\beta,\mu)$; equivalently, as the left derivative, with respect to $\mu$, of the pressure $p(\beta,\mu) =\sup_\rho(\rho \mu - f(\beta,\rho))$. Set 
\begin{equation*}
\rho_m:= \liminf_{\beta \to \infty} \rho(\beta,-m).
\end{equation*}

\begin{theorem} \label{thm:perco-can} 
	Let $R\geq r_1$, $m \in (0,M)$ and $R_m>\sqrt{d+3}\, \tilde r_m \geq \sqrt{d+3} r_0$. 
\begin{enumerate}
	\item %Let $R\geq r_1$. 
		Suppose that $v$ satisfies Assumption \ref{ass:integrable}. Then 
		%\begin{equation} 
			$$- \nu^* \leq \liminf_{\beta\to \infty} \beta^{-1} \log \rho_-(\beta;R).$$
		%\end{equation} 
		In addition, for every fixed $\nu > \nu^*$, sufficiently large $\beta$, $\rho = \exp(-\beta \nu)$, there is a unique measure $P$ in $\mathcal{G}_\theta(\beta,\rho)$. It has no infinite cluster, $P$-almost surely, and satisfies 
		\begin{equation} \label{eq:decay-can}
			P^\circ( |\mathcal{C}_\omega(0)|=k \bigr) = k \rho_k(P) \leq  C \rho \exp( - \beta c k) 
		\end{equation} 
		for suitable $C,c>0$ and all $k \in \N$. 
	\item %Let $R_m>\max(\sqrt{d+3}\, \tilde r_m, r_1)$. 
		Suppose that $v$ is continuous in $(r_0,r_1)$.Then 
		%$\rho_0 <\rho_m <\infty$ and 
	\begin{equation*}
		\limsup_{\beta \to \infty} \rho_+(\beta;R_m) \leq \rho_m. 
	\end{equation*}
\end{enumerate}
\end{theorem}
The constants $C$ and $c$ in Eq. \eqref{eq:decay-can} can be chosen uniform in regions of the form $\beta \geq \beta_\eps$, $\rho \leq \exp( - \beta (\nu^*+\eps))$. 

If in addition $v$ satisfies Assumptions \ref{ass:interparticle-lower} and \ref{ass:interparticle-upper}, then  $\rho_m$ is larger than the preferred ground state density $\rho_0$, and in particular, bounded away from zero; this follows from Theorem 3.2 in \cite{jansen12}. 
Moreover, we expect that for every $R\geq r_1$, as $\beta \to \infty$,  
\begin{equation*} 
	\lim_{\beta \to \infty} \beta^{-1} \log \rho_-(\beta;R) = - \nu^*,\quad 
		\lim_{\beta \to \infty} \rho_+(\beta;R) = \rho_0, 
\end{equation*} 
and for $\rho_-(\beta;R)< \rho <\rho_+(\beta;R)$ and very large $\beta$, there should be non-ergodic Gibbs measures with percolation probability strictly between $0$ and $1$. 
This is what happens for lattice gases (see Appendix \ref{app:lattice}). For continuum systems, we have no proof of this conjecture; we have, however, a result pointing in the right direction: 

\begin{prop}  \label{prop:soft-trans}
	Suppose that $v$ satisfies Assumptions \ref{ass:interparticle-lower} and \ref{ass:interparticle-upper}. Let $R\geq r_1$.  There are $\beta_0,\rho_0,C>0$ such that for all $\beta \geq \beta_0$, all $\rho \leq \rho_0$ and all $P \in \mathcal{G}_\theta(\beta,\rho)$, the following holds: if $\rho = \exp( - \beta \nu) > \exp(- \beta \nu^*)$, then 
\begin{equation*} 
	\forall K \in \N:\quad 
		\sum_{k=1}^K k \rho_k(P) \leq  \frac{C \rho \beta^{-1} \log \beta }{\nu^*- \nu}. 
\end{equation*}  
\end{prop}
Thus at densities above $\exp( - \beta \nu^*)$, the fraction of particles in finite-size clusters  is small.

\section{Topology on $\mathcal{P}$ and continuity properties} \label{sec:topo}

In this section we specify the topology on $\mathcal{P}$ that we use, recall some continuity properties of the functionals to be studied, and explain why the variational problems considered in this article have minimizers. We follow \cite{georgii-zessin93,georgii94}. 

Let $\mathcal{M}(\Omega)$ be the set of finite measures on $(\Omega,\mathcal{F})$. The 
topology $\tau_\mathcal{L}$ of \emph{local convergence} on $\mathcal{M}(\Omega)$ is defined as follows. Let $\mathcal{L}$ be the class of measurable functions $f:\Omega \to \R$ that are \emph{local} and \emph{tame}, i.e., $f\in \mathcal{L}$ if and only if there is  a Borel subset $B\subset \R^d$ and a constant $c>0$ such that $f$ is a function of $\omega_B$ alone and for all $\omega\in \Omega$, $|f(\omega)| \leq c(1+N_B(\omega))$.  Then $\tau_\mathcal{L}$ is the smallest topology with respect to which all maps of the form $P\mapsto \int_\Omega f(\omega)P(\dd \omega)$, $f\in\mathcal{L}$, are continuous. 

The following holds \cite{georgii-zessin93,georgii94}: 
\begin{itemize} 
	\item $\mathcal{P}$ and $\mathcal{P}_\theta$ are  closed subsets of $\mathcal{M}(\Omega)$. We endow them with the traces of the topology $\tau_\mathcal{L}$. 
	\item The Palm measure $\mathcal{P}_\theta \to \mathcal{M}(\Omega)$, $P\mapsto P^\circ$ is continuous. 
	\item The particle density $\mathcal{P}_\theta \to \R$, $P \mapsto \rho(P)$ is continuous. 
	\item The entropy density $\mathcal{P}_\theta \to \R\cup \{-\infty\}$, $P\mapsto S(P)$ is affine, upper semi-continuous, and has superlevel sets $\{S\geq - c\}$ that are compact and sequentially compact. 
	\item The energy density $\mathcal{P}_\theta \to \R \cup \{\infty\}$, $P \mapsto U(P)$ is lower semi-continuous. 
\end{itemize}

Furthermore, the cluster densities are continuous: 
\begin{lemma}
	For every $k\in \N$, the map $\mathcal{P}_\theta \to \R$, $P\mapsto \rho_k(P)$, is continuous. 
\end{lemma}

\begin{proof}
	Let $g_k(\omega):=\mathbf{1}(|\mathcal{C_\omega}(0)|=k)$. The function $g_k$ is local and bounded, thus in particular, tame. Therefore, by definition of $\tau_\mathcal{L}$,  $P\mapsto \mathcal \int_\Omega g_k \dd P$ is continuous. Since $P\mapsto \rho_k(P)$ is the composition of the latter map with the continous map $P\mapsto P^\circ$, it follows that 
	$P\mapsto \rho_k(P)$ is continuous. 
\end{proof}

Now we can easily check that the variational problem in Theorem \ref{thm:var} admits a minimizer. 

\begin{lemma} \label{lem:varmin}
	Fix $\beta,\rho>0$ and $(\rho_k)_{k\in \N} \in [0,\infty)^\N$. 
	Let $\mathcal{A}\subset \mathcal{P}_\theta$ be the set of measures satisfying $\rho(P)=\rho$ and $\rho_k(P)=\rho_k$, for every $k \in \N$. If 	$\mathcal{A}\neq \emptyset$, the function $\mathcal{A}\ni P \mapsto U(P) - \beta^{-1} S(P)$ has a minimizer. \end{lemma}	

\begin{proof}
	If $U(P) -\beta^{-1} S(P) =\infty$ for every $P \in \mathcal{A}$, there is nothing to show. If $U(P)-\beta^{-1} S(P)<\infty$ for some $P \in \mathcal{A}$, let 	
	$(P_n)$ be a minimizing sequence. The sequence  $(U(P_n)- \beta^{-1} S(P_n))_{n\in\N}$ is bounded from above and, because $U(P)$ is bounded from below, $S(P_n)$ is bounded from below too. Since the superlevel sets $\{S\geq -c\}$ are $\tau_\mathcal{L}$-sequentially compact, there is a subsequence $P_{n_j}$ converging to some $P\in\mathcal{P}_\theta$. The continuity of the maps $\rho(\cdot)$ and $\rho_k(\cdot)$ ensures that $P \in \mathcal{A}$, and the lower semi-continuity of $U$ and $-S$ shows that $P$ is a minimizer. 
\end{proof}

\section{Proof of Theorem \ref{thm:var}} \label{sec:var}

The main idea for the proof of Theorem \ref{thm:var} is to apply a large deviations principle for the stationary empirical field proven in \cite{georgii94,georgii-zessin93} and the contraction principle \cite[Section 4.2.1]{dembo-zeitouni-book}. Two complications stand in our way. First, the large deviations principle in \cite{georgii94,georgii-zessin93} was shown in the grand-canonical rather than the canonical ensemble. Second, the cluster size densities can only be expressed as functions of the stationary empirical field if we modify their definition and, loosely speaking, define them with periodic boundary conditions; this yields a modified variable $\vect{\rho}_\Lambda^\per$. 
In order to circumvent these difficulties, we proceed as follows: 
\begin{itemize}
	\item We show first that the large deviations principle in the canonical ensemble for $\vect{\rho}_\Lambda$  implies a large deviations principle in the grand-canonical ensemble (Lemma \ref{lem:gc1}). 
	\item We apply the contraction principle and show that  $\vect{\rho}_\Lambda^\per$ satisfies a large deviations principle with convex rate function (Lemma \ref{lem:contraction}).
	\item Next we show that (truncations of) $\vect{\rho}_\Lambda$ and $\vect{\rho}_\Lambda^\per$, in the grand-canonical ensemble, are \emph{exponentially equivalent} \cite[Section 4.2.2]{dembo-zeitouni-book}; this follows from Ruelle's superstability bounds \cite{ruelle70}. As a consequence, the grand-canonical rate functions for $\vect{\rho}_\Lambda$ and $\vect{\rho}_\Lambda^\per$ must be equal (Lemma \ref{lem:bc}). 
	\item Taking Legendre transforms, we deduce the desired identity  for the canonical rate function (Lemma \ref{lem:legendre}). 
\end{itemize}	
	
For the purpose of this section it is most convenient to work with measures that are not normalized and to suppress the $\beta$-dependence in the notation.  Let 
\begin{equation*}
	Q_{N,\Lambda}^\can(A) := \frac{1}{N!} \int_{A} \exp(- \beta U(x_1,\ldots,x_N)) \dd x_1 \cdots \dd x_N
\end{equation*} 
be a measure on $\Lambda^N$ with total mass $Z_\Lambda(\beta,N)$.  The image of $Q_{N,\Lambda}^\can$ under  the map $\Lambda^N\to\Omega$, $(x_1,\ldots,x_N)\mapsto \{x_1,\ldots,x_N\}$ is a measure on $(\Omega,\mathcal{F})$, for which we use the same letter $Q_{N,\Lambda}^\can$. Furthermore define
\begin{equation*} 
	Q_{\mu,\Lambda} = \delta_\emptyset + \sum_{N=1}^\infty z^N  Q_{N,\Lambda}^\can,\quad z = \exp(\beta \mu).
	\end{equation*} 
We use the same letter for the measure on $\Omega$ and the measure on disjoint unions $\dot \cup_{N\geq 0} \Lambda^N$. $\Lambda^0$ is a dummy space corresponding to $\omega = \{\emptyset\}$: the event that there is no point at all has $Q_{\mu,\Lambda}$-measure $1$.  Remember that  $\R_+^\N$ is equipped with the product topology and the  corresponding Borel $\sigma$-algebra. 

\begin{lemma}\label{lem:gc1}
	Under $Q_{\mu,\Lambda}$ as $|\Lambda|\to \infty$, $N\to \infty$ at fixed $\beta >0$ and $\mu \in \R$,  
	the cluster size distribution $\vect{\rho}_\Lambda:= (\rho_{k,\Lambda})_{k \in \N}$ satisfies a large deviations principle with speed $\beta|\Lambda|$ and rate function 
		\begin{equation} \label{eq:jvar}
				J_{\beta,\mu}\bigl( (\rho_k)_k \bigr) = \inf_{\rho> 0} \Bigl( f\bigl(\beta,\rho, (\rho_k)_k\bigr)	
							- \mu \rho\Bigr).
		\end{equation}
\end{lemma}
We recall that $f(\beta,\cdot,\cdot): \R_+\times \R_+^\N\to \R$ is a lower semi-continuous, convex function, defined in all of $\rho >0$ and not only $\rho \in(0,\rho_\mathrm{cp})$. When  $\rho>\rho_\mathrm{cp}$ or $\sum_{k+1}^\infty k \rho_k >\rho$,  it takes the value $\infty$ \cite{jkm11}. 

\begin{proof}[Proof of Lemma \ref{lem:gc1}] 
	We consider Eq. \eqref{eq:jvar} as the definition of a function $J_{\beta,\mu}$ and show that $J_{\beta,\mu}$ is a rate function for $(\rho_{k,\Lambda})$. Set $z:=\exp(\beta \mu)$. 
	
	\emph{Lower bound.}  Let $\mathcal{O}\subset \R_+^\N$ be an open set 
	For every $N$ and $\Lambda$, 
		\begin{equation*}
				Q_{\mu,\Lambda}( (\rho_{k,\Lambda})_k \in \mathcal{O}) \geq z^N Q_{N,\Lambda}^\can((\rho_k)_k \in \mathcal{O}).
			\end{equation*}
			Fix $\rho>0$ and apply the previous inequality to $N,\Lambda$ with $N\to \infty$, $|\Lambda| \to \infty$ such that $N/|\Lambda|\to \rho$. We get 
				\begin{equation*}
						\liminf \frac{1}{\beta |\Lambda|} \log Q_{\mu,\Lambda} ( (\rho_{k,\Lambda})_k \in \mathcal{O}) 
								\geq \mu \rho - \inf_{(\rho_k) \in \mathcal{O}} f(\beta,\rho,(\rho_k)). 
				\end{equation*}
				Since the inequality holds for every $\rho>0$, we can take the supremum over $\rho>0$ on the right-hand side, and to conclude note 
				\begin{align*}
					\sup_{\rho>0} \Bigl( \mu \rho -\inf_{(\rho_k) \in \mathcal{O}} f(\beta,\rho,(\rho_k)) \Bigr) 
					&  = \sup_{\rho>0} \sup_{(\rho_k) \in \mathcal{O}} \Bigl( \mu \rho - f(\beta,\rho, (\rho_k)) \Bigr) \\
					&  = \sup_{(\rho_k) \in \mathcal{O}} \sup_{\rho>0} \Bigl( \mu \rho - f(\beta,\rho, (\rho_k)) \Bigr) \\
					& = - \inf_{(\rho_k) \in \mathcal{O}} J_{\beta,\mu}((\rho_k)). 
				\end{align*}
				
		\emph{Upper bound.} 
			We use ideas explained in \cite[Section 3.4.5]{ruelle-book}. 
Let $\mathcal{A} \subset \R_+^\N$ be a closed set. 
			For every $\rho>0 $ as  $|\Lambda|\to \infty$ and $N/|\Lambda|\to \rho$, 
			\begin{equation} \label{eq:ls}
				\limsup \frac{1}{\beta |\Lambda|} \ln Q_{N,\Lambda} ^\can( (\rho_{k,\Lambda})_k \in \mathcal{A}) 
								\leq - \inf_{(\rho_k) \in \mathcal{A}} f(\beta,\rho,(\rho_k)). 
			\end{equation}	
			Let $b<\infty$ be the stability constant from Eq. \ref{eq:superstable}.
			When $N/|\Lambda|\to 0$, note that 
		$$Q_{N,\Lambda}^\can((\rho_{k,\Lambda})\in\mathcal{A}) \leq Z_\Lambda(\beta,N) 
			\leq \exp(\beta N b) \frac{|\Lambda|^N}{N!} $$
			from which we get 
			\begin{equation*}
				\limsup_{|\Lambda|\to \infty} \frac{1}{|\Lambda|} \log Q_{N,\Lambda}^\can((\rho_k)\in\mathcal{A}) 
				\leq  \limsup_{|\Lambda|\to \infty} \frac{1}{|\Lambda|} 
					\Bigl( \beta N b - N \log \frac{N}{|\Lambda|e}\Bigr)  = 0.
			\end{equation*} 	
			Thus Eq.~\eqref{eq:ls} holds for all $\rho \geq 0$, provided we read the right-hand side as $0$ when $\rho =0$.  Fix $\rho_0>0$. We claim that Eq. \eqref{eq:ls} holds uniformly in $\rho\in [0,\rho_0]$. More precisely,	for every $\epsilon>0$, there is a $\delta>0$ such that: for all $\rho \in [0,\rho_0]$ and all $N,\Lambda$ satisfying  $|\Lambda|\geq 1/\delta$ and $|N/|\Lambda| - \rho| \leq \delta$, we have 
			\begin{equation} \label{eq:ls-unif}
				 \frac{1}{\beta |\Lambda|} \log Q_{N,\Lambda}^\can( (\rho_{k,\Lambda})_k \in \mathcal{A}) 
								\leq \eps - \inf_{(\rho_k) \in \mathcal{A}} f(\beta,\rho,(\rho_k)). 
			\end{equation}			
			Indeed, if this was not the case, we could find  $\epsilon >0$, sequences $(N_j), (\Lambda_j)$ and a $\rho \in [0,\rho_0]$ such that $|\Lambda_j| \geq j$, $|N_j/|\Lambda_j| - \rho|\leq 1/j$ and
				\begin{equation*} 
				 \frac{1}{\beta |\Lambda_j|} \ln Q_{N_j,\Lambda_j} ^\can( (\rho_{k,\Lambda})_k \in \mathcal{A}) 
								\geq \eps - \inf_{(\rho_k) \in \mathcal{A}} f(\beta,\rho,(\rho_k)),
			\end{equation*}
			contradicting Eq. \eqref{eq:ls}; this proves the claim. Now because of the uniformity, for every $\epsilon>0$ and sufficiently large $|\Lambda|$, 
				\begin{align*}
					&\frac{1}{\beta|\Lambda|} \log \Bigl(\sum_{N=1}^{\lfloor \rho_0 |\Lambda|\rfloor} z^N Q_{N,\Lambda}\bigl( (\rho_{k,\Lambda}) \in \mathcal{A} \bigr) \Bigr) \\
					& \quad \leq \frac{\log (2+\rho_0|\Lambda|)}{\beta|\Lambda|} 
							+ \epsilon + \sup_N\Bigl( \mu \frac{N}{|\Lambda|} - \inf_{(\rho_k) \in \mathcal{A}} f(\beta,N/|\Lambda| ,(\rho_k)) \Bigr) \\
					& \quad \leq o(1) + \eps + \sup_{\rho>0} \Bigl( \mu \rho - \inf_{(\rho_k) \in \mathcal{A}} f(\beta,\rho, (\rho_k)) \Bigr) \\
					&\quad = o(1) + \eps - \inf_{ (\rho_k) \in \mathcal{A}} J_{\beta,\mu}((\rho_k)). 
				\end{align*}
	If $(\rho_k)=\vect{0}\notin \mathcal{A}$, the $\delta_\emptyset$ term corresponding, formally, to $N=0$, does not contribute to $Q_{\mu,\Lambda}((\rho_{k,\Lambda})\in\mathcal{A})$. If $\vect{0}\in \mathcal{A}$, the $\delta_\emptyset$-term contributes a term $1$ inside the logarithm, and we want to check that 
	$\lim |\Lambda|^{-1} \log 1 \leq - \inf_\mathcal{A} J_{\beta,\mu}$. 	
	To this aim note that 
	\begin{equation*}
		\inf_\mathcal{A} J_{\beta,\mu}((\rho_k)) \leq \inf_{\rho>0}\bigl( f(\beta,\rho,\vect{0}) - \mu \rho\bigr) \leq  \liminf_{\rho \to 0} \bigl( f(\beta,\rho, \vect{0}) - \mu \rho) =0.
	\end{equation*}
	The last equality is  shown by choosing a connected reference configuration $(x_1^0,\ldots,x_N^0)$ which is such that (i) for suitable $r>0$, every $\vect{x} \in \cup_{\pi \in \mathcal{S}_N} \times_{k=1}^N B(x_{\pi(k)}^0,r)$ is connected, (ii) on this set the energy is upper bounded by $CN$ for suitable $C>0$,  and (iii) the balls are disjoint.  Integrating in the neighborhood of $\vect{x}^0$ given by  $\times_{i=1}^N B(x_i,r)$ (and permutations of this set) yields the bound 
	\begin{equation*} 
		- \beta f(\beta,\rho,\vect{0}) 
			%\geq - \beta \lim_{N,|\Lambda|\to \infty}  \frac{1}{|\Lambda|} \log \frac{N! |B(0,r)|^N \exp( - \beta CN )}{N!} 
			\geq - \beta \rho \log \Bigl( |B(0,r)|e^{1-\beta C} \Bigr). 
	\end{equation*} 
	The right-hand side goes to $0$ as $\rho \to 0$. Thus we need not worry about the  contribution from $N=0$.  
%	hence $\limsup 1/|\Lambda|\leq \inf_\mathcal{A} J_{\beta,\mu}$.
	
	 It remains to estimate the terms from $N\geq  \rho_0 |\Lambda|$. Choose $\rho_0$  large enough so that $z \exp(\beta b+1) / \rho_0 <1 /2$. Remember $N! \geq (N/e)^N$. Then 		
				\begin{align*}
					\sum_{N \geq \rho_0|\Lambda|} z^N Q_{N,\Lambda}^\can( (\rho_{k,\Lambda}) \in \mathcal{A}) & \leq \sum_{N \geq \rho_0|\Lambda|} z^N \frac{|\Lambda|^N}{N!} e^{\beta b N} \\
						&\leq \sum_{N \geq \rho_0|\Lambda|} \Bigl( \frac{z |\Lambda| \exp(\beta b+1)}{N} \Bigr)^N \leq 2 \exp\bigl( - \rho_0 |\Lambda| \ln 2\bigr).
				\end{align*}
		 If $\inf_\mathcal{A} J_{\beta,\mu} <\infty$, we can choose $\rho_0$ large enough so that $\rho_0 \log 2 > \beta \inf_{\mathcal{A}} J_{\beta,\mu}$, and we obtain 
			\begin{equation} \label{eq:ub}
					\limsup \frac{1}{\beta|\Lambda|} \log Q_{\mu,\Lambda} \bigl( (\rho_{k,\Lambda}) \in \mathcal{A} \bigr) \leq - \inf_{(\rho_k) \in \mathcal{A}} J_{\beta,\mu}\bigl( (\rho_k) \bigr). 
			\end{equation}
	If $\inf_\mathcal{A} J_{\beta,\mu}=\infty$, we have $f(\beta,\rho,(\rho_k))=\infty$ for all $\rho >0$ and $(\rho_k) \in \mathcal{A}$. But this implies that $Q_{N,\Lambda}^\can( (\rho_{k,\Lambda})_k \in \mathcal{A})=0$ for all $N,\Lambda$. Indeed, suppose by contradiction that there is a $N \in \N$ and a box $\Lambda =[0,L]^d$ such that $Q_{N,\Lambda}^\can ((\rho_{k,\Lambda})_k) \in \mathcal{A}) >0$. 
	For $n\in \N$, let $L_n:= n(L+R)$, $\Lambda_n:=[0,L_n]^d$ and $N_n:=n^d N$.  	
	Then 
	\begin{equation*} 
		Q_{N_n,\Lambda_n}^\can\bigl( (\rho_{k,\Lambda}) \in \mathcal{A} )\bigr) \geq 
		\Bigl(  Q_{N,\Lambda}^\can( \bigl(\rho_{k,\Lambda}) \in \mathcal{A} )\bigr) \Bigr)^{n^d} 
	\end{equation*} 
 	and 
 	\begin{equation*} 
 	 	- \beta \inf_{(\rho_k) \in \mathcal{A}}  f(\beta,\rho,(\rho_k))
 	 			\geq -  \frac{1}{(L+R)^d} \log Q_{N,\Lambda}^\can\bigl( (\rho_{k,\Lambda}) \in \mathcal{A} \bigr) >- \infty,
 	 \end{equation*} 
 	 contradiction. Thus $Q_{N,\Lambda}^\can( (\rho_{k,\Lambda})_k \in \mathcal{A})=0$ 
 	 for all $N$ and $\Lambda$, whence $Q_{\mu,\Lambda}( (\rho_{k,\Lambda})\in \mathcal{A}) =0$; Eq. \eqref{eq:ub} holds trivially. 
\end{proof}

Next we define the stationary empirical field and the modified cluster size distribution $\vect{\rho}_\Lambda^\per$. 
%we prove a large deviations principle for cluster size distributions with periodic boundary conditions, defined in such a way that $\rho_{k,\Lambda}^\per$ are continuous functions of the stationary empirical process $R_{\Lambda,\omega}$. This allows us to apply a contraction principle. Let us first recall the notion of stationary empirical process.
For $\omega \in \Omega$ and $\Lambda = [0,L]^d$, let $\omega_\Lambda^\per:=\cup_{\vect{k\in \Z^d}} \theta_{L\vect{k}} (\omega \cap \Lambda)$ be the periodic continuation of $\omega \cap \Lambda$. 
%		If we consider $\Lambda$ as a torus, we can shift every configuration $\omega_\Lambda \in \Omega_\Lambda$ along a vector $x\in \Lambda$ and obtain a new configuration on $\Lambda$, which we denote $\theta_x \omega_\Lambda^\per$. 
The \emph{translation invariant empirical field} is 
	\begin{equation*}
			R_{\Lambda,\omega} = \frac{1}{|\Lambda|} \int_{\Lambda} \delta_{\theta_x \omega^\per_\Lambda} \dd x.
		\end{equation*} 
	The Palm measure of $R_{\Lambda,\omega}$ is 
		\begin{equation*}
				R_{\Lambda,\omega}^\circ = \frac{1}{|\Lambda|} \sum_{x \in \omega\cap \Lambda} \delta_{\theta_x \omega_\Lambda^\per}.
		\end{equation*}
%	To see this, note that 
%	\begin{align*}
%		\int_\Omega \sum_{y\in \tilde \omega} f(y,\theta_y \tilde \omega) R_{\Lambda,\omega}(\dd \tilde \omega) &= \frac{1}{|\Lambda|} \int_\Lambda \dd x \sum_{y\in \theta_x \omega_\Lambda^\per} f(y,\theta_{y} \theta_x \omega_\Lambda^\per) \\
%		& = \frac{1}{|\Lambda|} \int_\Lambda \dd y \sum_{x \in \theta_y  \omega_\Lambda^\per} f(y,\theta_{y} \theta_x \omega_\Lambda^\per)\\
%		& = \frac{1}{|\Lambda|} \int_\Lambda \dd y \sum_{\hat x \in  \omega_\Lambda^\per} f(y,\theta_{\hat x}  \omega_\Lambda^\per).
%	\end{align*}
\begin{remark}
		The stationary empirical field associates with every configuration $\omega$ a probability measure supported on configurations with the same relative coordinates as $\omega \cap \Lambda$, but randomized center of mass.
		For example, in dimension $d=1$, if $\omega$ consists of two particles $0,\eps$, then $R_{\Lambda,\omega}$ describes configurations $\{x,x+\eps\}$ with $x$ having uniform Lebesgue density $1/L$. 
\end{remark} 
We define $k$-cluster densities with periodic boundary conditions as
\begin{equation*}
	k \rho_{k,\Lambda}^\per(\omega):=\int_\Omega \mathbf{1}\bigl( |\mathcal{C}_{\tilde \omega}(0)|=k \bigr)R_{\Lambda,\omega}^\circ(\dd \tilde \omega)
%	= \frac{1}{|\Lambda|} \sum_{x\in \omega \cap \Lambda} 
%				\mathbf{1}\bigl( |\mathcal{C}_{\omega_\Lambda^\per}(x)|=k \bigr)
	= \frac{1}{|\Lambda|}\Bigl|\Bigl\lbrace x\in \omega \cap \Lambda \,\Big|\, |\mathcal{C}_{\omega_\Lambda^\per}(x)|=k \Bigr\rbrace \Bigr|.
\end{equation*}
We have $\rho_{k,\Lambda}^\per(\omega) = \rho_k(R_{\Lambda,\omega})$; compare with  Eq.~\eqref{eq:rhokp}. Set $\vect{\rho}_\Lambda(\omega):= (( \rho_{k,\Lambda}^\per(\omega))_{k\in \N}$. 

\begin{lemma} \label{lem:contraction}
	 Under $Q_{\mu,\Lambda}$, $\vect{\rho}_{\Lambda}^\per$ satisfies a large deviations principle with rate function 
		\begin{equation*}
			J_{\beta,\mu}^\per\bigl( (\rho_k) \bigr) = \inf \{ U(P) - \beta^{-1} S(P) - \mu \rho(P) \mid P \in \mathcal{P}_\theta,\, \forall k \in \N:\, \rho_k(P)= \rho_k \}.
		\end{equation*}
\end{lemma}

\begin{proof}
		Under $Q_{\mu,\Lambda}$, the stationary empirical field satisfies a large deviations principle with rate function $U(P) - \beta^{-1} S(P)  - \mu \rho(P)$ \cite{georgii94}, with compact sublevel sets. Since $(\rho_{k,\Lambda}^\per(\omega))_{k\in \N}$ can be written as a continuous function of the stationary empirical field, we can apply the contraction principle \cite[Section 4.2.1]{dembo-zeitouni-book}	and the result follows. 
\end{proof}

\begin{lemma}\label{lem:bc}
	For every $\beta >0$ and $\mu \in \R$, we have
	$J_{\beta,\mu} = J_{\beta,\mu}^\per$. 
\end{lemma}

\begin{proof}%[Proof of Lemma \ref{lem:bc}]
	When $x\in \Lambda$ has distance $>kR$ to the boundary $\partial \Lambda$, we have $\mathcal{C}_{\omega_\Lambda}(x) = \mathcal{C}_{\omega_\Lambda^\per}(x)$. Set 
	$\partial_{kR}\Lambda:=\{x\in \Lambda\mid \dist(x,\partial\Lambda)\leq kR\}$. 
	Then 
	\begin{equation*}
		\bigl|k \rho_{k,\Lambda}(\omega) - k \rho_{k,\Lambda}^\per(\omega)\bigr|
			\leq \frac{|\omega \cap \partial_{kR} \Lambda|}{|\Lambda|}
	\end{equation*}
	and for every $\epsilon >0$, 
 	\begin{equation*}
		P_{\mu,\Lambda}
			\Bigl( \bigl| \rho_{k,\Lambda}(\omega) -  \rho_{k,\Lambda}^\per(\omega)\bigr|\geq \eps \Bigr) \leq 
			P_{\mu,\Lambda}\Bigl( N_{\partial_{k R}\Lambda} \geq k \eps |\Lambda| \Bigr). 
	\end{equation*}
	Here $P_{\mu,\Lambda}:= (\Xi_{\Lambda}(\beta,\mu))^{-1} Q_{\mu,\Lambda}$ is the grand-canonical probability measure. 
%	When the pair interaction has a hard core, we deduce that 
%	\begin{equation*}
%		\bigl| \rho_{k,\Lambda}(\omega) -  \rho_{k,\Lambda}^\per(\omega)\bigr|
%			\leq \frac{ 2 d  R}{|B(0,r_\mathrm{hc})|\, L}
% 	\end{equation*}
 %which for every fixed $k$, goes to $0$ as $L\to \infty$. When $v$ has no hard core, we use
	Because of the superstability of $v$, we know  \cite{ruelle70} that for some suitable  $\xi >0$ such that all $n$-point correlation functions $\rho_\Lambda(x_1,\ldots,x_n)$ are bounded in absolute value by $\xi^n$; $\xi$ depends on $\beta$ and $\mu$ but is independent of $\Lambda$ or $n$.  Integrating, we get bounds for factorial moments as
	\begin{equation*}
			E_{\mu,\Lambda}
					 \Bigl[ N_{\partial_{kR} \Lambda}\bigl(N_{\partial_{kR} \Lambda} -1) \cdots 
					 		\bigl(N_{\partial_{kR} \Lambda} -m+1) \Bigr] \leq \bigl( \xi |\partial_{kR} \Lambda| \bigr)^m.
	\end{equation*} 	
	Let $G(z)=\sum_{n=0}^\infty P_{\mu,\Lambda}(N_{\partial_{kR}(\omega) \Lambda}=n) z^n$ be the probability generating function of $N_{\partial_{kR} \Lambda}(\omega)$. A Taylor expansion around $z=1$ yields  
	\begin{align*} 
			P_{\mu,\Lambda}(N_{\partial_{kR} \Lambda} (\omega)=n) 
					& = \frac{1}{n!} G^{(n)}(0) = \frac{1}{n!}  \frac{\dd^n}{\dd z^n} \sum_{k=0}^\infty \frac{G^{(k)}(1)}{k!} (z-1)^k \Big|_{z=0}\\
			& = \frac{1}{n!}  \sum_{k=n}^\infty \frac{G^{(k)}(1)}{(k-n)!} (-1)^{k-n} 
				\leq \frac{1}{n!} \sum_{k=n}^\infty  \frac{\bigl( \xi |\partial_{kR} \Lambda| \bigr)^k}{(k-n)!} \\
				&= \frac{\bigl( \xi |\partial_{kR} \Lambda|\bigr)^n}{n!} e^{ \xi |\partial_{kR} \Lambda|} \leq \Bigl( \frac{e\xi |\partial_{kR} \Lambda|}{n} \Bigr)^n e^{ \xi |\partial_{kR} \Lambda|}. 
	\end{align*}
	For $M\in \N$ such that $M>2 e \xi |\partial_{kR} \Lambda|$, we have 
	\begin{equation*} 
				 	P_{\mu,\Lambda}(N_{\partial_{kR} \Lambda}(\omega) \geq M) 
						 \leq 2 \Bigl( \frac{e\xi |\partial_{kR} \Lambda|}{M}  \Bigr)^M e^{\xi|\partial_{kR}\Lambda|}.
	\end{equation*}
	Applying the inequality to $M=k \eps |\Lambda|$ we obtain, for sufficiently large $|\Lambda|=L^d$, 
	\begin{equation*}
		 P_{\mu,\Lambda}\Bigl( \bigl| \rho_{k,\Lambda}(\omega) -  \rho_{k,\Lambda}^\per(\omega)\bigr|\geq \eps \Bigr)  
		 	\leq 2 \Bigl( \frac{ 2 d e \xi R}{L} \Bigr)^{k \eps L^d} e^{2 d \xi k R L^{d-1}}.
	\end{equation*}
%	Next, we estimate 
%	\begin{align*}
%		P_{\beta,\mu}\Bigl( \bigl| \rho_{k,\Lambda}(\omega) -  \rho_{k,\Lambda}^\per(\omega)\bigr|\geq \eps \Bigr)& \leq 
%			P_{\mu,\Lambda}\Bigl( N_{\partial_{k R}\Lambda} \geq k \eps |\Lambda| \Bigr) \\
%				& \leq 2 \Bigl( \frac{e \xi |\partial_{kR} \Lambda|}{k \eps |\Lambda|} \Bigr)^{k \eps |\Lambda|}
%						e^{ \xi |\partial_{kR} \Lambda|} \\
%				& \leq 2 \Bigl( \frac{ 2 d e \xi R}{L} \Bigr)^{k \eps L^d} e^{2 d \xi k R L^{d-1}}.
%	\end{align*}
	It follows that for every fixed $K\in \N$ and $\eps>0$, 
	\begin{equation*} 
		\lim_{|\Lambda| \to \infty} \frac{1}{|\Lambda| ^d} \log P_{\mu,\Lambda} \Bigl( \sum_{k=1}^K |\rho_{k,\Lambda} (\omega) - \rho_{k,\Lambda}^\per(\omega) \bigr| \geq \eps \Bigr) 
			= - \infty, 
	\end{equation*}  
	and the same identity holds with $P_{\mu,\Lambda}$ replaced with $Q_{\mu,\Lambda}$ (note $\Xi_\Lambda(\beta,\mu)\geq 1$). 
	As a consequence, the random variables $(\rho_{1,\Lambda},\ldots,\rho_{K,\Lambda})$ and 
	$(\rho_{1,\Lambda}^\per,\ldots,\rho_{K,\Lambda}^\per)$ are \emph{exponentially equivalent} in the sense of \cite[Section 4.2.1]{dembo-zeitouni-book}. The contraction principle together with Lemma \ref{lem:contraction}  and results from \cite{jkm11} tell us that they satisfy large deviations principles; because of exponential equivalence, their rate functions must be equal \cite[Theorem 4.2.3]{dembo-zeitouni-book}. The Dawson-G{\"a}rtner theorem \cite[Theorem 4.6.1]{dembo-zeitouni-book} allows us to recover the rate function of the full vectors ($k \in \N$) from the rate functions of the truncated vectors ($1\leq k \leq K$)  and it follows that $J_{\beta,\mu}= J_{\beta,\mu}^\per$. 
\end{proof}

\begin{lemma}\label{lem:legendre}
	For every $\beta>0$, $\rho \in (0,\rho_\mathrm{cp})$ and $(\rho_k) \in \R_+^\N$, we have 
	\begin{equation*}
			f(\beta,\rho,(\rho_k))= \inf \{U(P)-\beta^{-1}S(P) \mid P\in \mathcal{P}_\theta,\ 
				\rho(P) = \rho,\  \forall k \in \N:\, \rho_k(P) = \rho_k \}. 
	\end{equation*}
\end{lemma}

\begin{proof}[Proof of Lemma \ref{lem:legendre}]
	Fix $\beta>0$ and $(\rho_k)_{k\in \N} \in \R_+^\N$.  For $\rho >0$, set 
	\begin{equation*} 
		h(\rho):= \inf\{ U(P) - \beta^{-1} S(P) \mid P \in \mathcal{P}_\theta, \, \rho(P) = \rho,\, 
			\forall k \in \N:\, \rho_k(P) =\rho_k \}
	\end{equation*}
	and $g(\rho):= f(\beta,\rho, (\rho_k))$ for $\rho \in (0,\rho_\mathrm{cp})$, $g(\rho):= \infty$ for $\rho < \rho_\mathrm{cp})$; the value of $g(\rho_\mathrm{cp})$ can be determined in such a way that $g$ is lower semi-continuous and convex in $(0,\infty)$. 
	Lemma  \ref{lem:bc} implies that $h$ and $g$ have the same Legendre transforms, 
	\begin{equation*} 
		\forall \mu \in \R:\quad 	\sup_{\rho>0 } (\rho \mu - h(\rho) ) = 	\sup_{\rho>0} (\rho \mu - g(\rho) ). 
	\end{equation*}
	Since the functions $h$ and $g$ are lower semi-continuous in $(0,\infty)$ and convex \cite{georgii94,jkm11}, it follows that $h(\rho) = g(\rho)$ for all $\rho \in (0,\rho_\mathrm{cp})$. 
\end{proof}

\begin{proof}[Proof of Theorem \ref{thm:var}] 
	Theorem \ref{thm:var} is an immediate consequene of  Lemmas \ref{lem:legendre} 
	and \ref{lem:varmin}.
\end{proof}

\section{Proof of Theorem \ref{thm:equiv} and Prop. \ref{prop:perco}}

\begin{proof}[Proof of Theorem \ref{thm:equiv}] 
	\emph{(2)$\Rightarrow$(1).}
	Suppose that there is a shift-invariant Gibbs measure $P$ with $\rho(P) = \rho$ and $\rho_k(P) = \rho_k$, for all $k$. By Theorem \ref{thm:var}, 
	$$\min f(\beta,\rho,\cdot) \leq f(\beta,\rho, (\rho_k)) \leq U(P) - \beta^{-1} S(P).$$
	On the other hand, let $(\rho'_k)$ with $\sum_{k=1}^\infty k \rho'_k \leq \rho$. By Theorem \ref{thm:var}, we find $P'\in \mathcal{P}_\theta$ such that $\rho(P')=\rho$, 
	$\rho_k(P')=\rho_k'$ for all $k$, and $f(\beta,\rho,(\rho'_k))= U(P')- \beta^{-1} S(P')$. 
	The Gibbs variational principle implies 
	\begin{equation*} 
		f(\beta,\rho,(\rho'_k)) - \mu \rho = U(P')- \beta^{-1} S(P') - \mu \rho(P') 
				\geq U(P) - \beta^{-1} S(P) - \mu \rho(P).
	\end{equation*}  
	Minimization over $(\rho'_k)$ yields (remember $\rho(P)=\rho$) 
	\begin{equation*} 
		\min f(\beta,\rho,\cdot) \geq U(P) - \beta^{-1} S(P).
	\end{equation*}		
	It follows that $\min f(\beta,\rho,\cdot) = f(\beta,\rho,(\rho_k)) = U(P) - \beta^{-1} S(P)$, and $(\rho_k)$ is a minimizer of $f(\beta,\rho,\cdot)$. 
	
	\emph{(1)$\Rightarrow$(2).}	Conversely, let $(\rho_k)$ be a minimizer of $f(\beta,\rho,\cdot)$. Thus $f(\beta,\rho,(\rho_k)) = f(\beta,\rho)$. By Theorem \ref{thm:var}, we can find $P\in \mathcal{P}_\theta$ such that $\rho(P) = \rho$, $\rho_k(P) = \rho_k$ for all $k$, 
	and $f(\beta,\rho) = f(\beta,\rho,(\rho_k))= U(P) - \beta^{-1} S(P)$. 
We know  that $f(\beta,\rho)= \sup_\mu (\mu \rho - p(\beta,\mu))$, and for every $\rho \in (0,\rho_\mathrm{cp})$, the supremum is actually attained. (This is because $\mu \mapsto \rho \mu - p(\beta,\mu)$ is concave and continuous in $\R$ with limits $-\infty$ as $\mu \to \pm \infty$; note $\lim_{\mu \to \infty} \partial_\mu p(\beta,\mu) = \rho_\mathrm{cp}$.) 
	Thus let $\mu\in \R$ such that $f(\beta,\rho) = \rho \mu - p(\beta,\mu)$. 
	We get $U(P)- \beta^{-1} S(P) - \mu \rho(P) = - p(\beta,\mu)$ and as a consequence of the Gibbs variational principle, we find that $P\in \mathcal{G}(\beta,\mu)$. 
\end{proof}

\begin{proof}[Proof of Prop. \ref{prop:perco}]
	 \emph{(2) $\Rightarrow$ (3)}  By definition of our probability space $\Omega$, for every configuration $\omega$ and every bounded set $A$ the number of particles in $A$ is finite. Therefore, if $\omega$ has a cluster with infinitely many particles, this cluster has infinite diameter. 
	 
	 \emph{(3) $\Rightarrow$ (2)} If in a configuration $\omega$ there is a cluster with infinite diameter, then this cluster must contain infinitely many particles -- otherwise it would have diameter bounded by $R$ times the cardinality of the cluster. 
	 
	 \emph{(1) $\Leftrightarrow$ (2)} 
	 Let  
		\begin{equation*}
		g(\omega):= \mathbf{1}(|\mathcal{C}_\omega(0)|=\infty) = 
			1 - \lim_{n\to \infty} \sum_{k=1}^n \mathbf{1}(|\mathcal{C}_\omega(0)=k). 
		\end{equation*}	
	The function is the limit of local, measurable functions and therefore measurable. 
	Eq. \eqref{eq:gpalm} gives 
	\begin{equation*}
		\int_\Omega g(\omega) P^\circ(\dd \omega) = \int_\Omega \sum_{x\in \omega \cap [0,1]^d} \mathbf{1}\bigl(|\mathcal{C}_\omega(x)|=\infty \bigr) P(\dd \omega).
	\end{equation*}
	Thus 
	\begin{equation} \label{eq:deficit-perco}
		\rho(P)- \sum_{k=1}^\infty k \rho_k(P) = \int_\Omega \Bigl| \bigl \{ x\in \omega \cap [0,1]^d \mid |\, |\mathcal{C}_\omega(x)|=\infty \bigr\} \Bigr| P(\dd \omega)
	\end{equation}
	is the expected number of particles in $[0,1]^d$ belonging to an infinite cluster. 
	If $\rho(P) - \sum_{k=1}^\infty k \rho_k(P)>0$, it follows right away that with positive probability, there is an infinite cluster. If $\rho(P)-\sum_{k=1}^\infty k \rho_k(P)=0$, using shift-invariance, we see that for every unit cube $C(\vect{k})$, $\vect{k} \in \Z^d$ (see p. \pageref{eq:superstable}) the probability that the cube intercepts an infinite cluster is zero; it follows that the probability that there is an infinite cluster vanishes. 
\end{proof}

\section{Percolation properties of Gibbs measures}  \label{sec:perco}

In this section we prove Theorems \ref{thm:perco-gc} and \ref{thm:perco-can} and Proposition \ref{prop:soft-trans}.  The proofs are a combination of results from \cite{muermann75, zessin08, pechersky-yambartsev09, jkm11},  what is known from cluster expansions \cite{ruelle-book,ruelle70}, and equivalence of ensembles as in \cite{georgii95}.

\subsection{Grand-canonical ensemble} 
For the proof of percolation, it is convenient to discretize space. 
Fix $\ell>0$. For $\vect{k} \in \Z^d$, let  $C(\vect{k}):= [k_1 \ell, (k_1+1) \ell)\times\cdots \times  [k_N \ell, (k_N+1)\ell)$. Let $\mathcal{C}$ be the collection of cubes $C(\vect{k})$, $\vect{k}\in \Z^d$. We say that two cubes are nearest neighbors
if their centers have Euclidean distance $\ell$. A collection $\mathcal{R} \subset \mathcal{C}$ of cubes is \emph{connected} if any two cubes in $\mathcal{R}$ can be joined by a path $(C_1,\ldots,C_n)$ of cubes in $\mathcal{R}$ such that  for every $j$, the cubes $C_{j}$ and $C_{j+1}$ are nearest neighbors. The following lemma is a variant of well-known contour arguments \cite[Section 5.3]{ruelle-book}. 

\begin{lemma}\label{lem:cells}  
	Let $P$ be a  probability measure on $\{0,1\}^\mathcal{C}$. There is a  constant $\alpha_d>0$ such that the following holds: if for some $\alpha >\alpha_d$, some  $A>0$, and all  connected subsets $\mathcal{R}\subset \mathcal{C}$, 
	\begin{equation*}
		P( \forall C\in \mathcal{R}: \omega_C = 0) \leq A \exp(-\alpha |\mathcal{R}|),
	\end{equation*}
	then $P$-almost surely, there is an infinite connected set $\mathcal{W}\subset \mathcal{C}$ such that $\omega_C =1$ for all $C\in \mathcal{W}$. 
\end{lemma} 
We refer to cubes $C$ with $\omega_C=0$ as \emph{empty}, and a cube with  $\omega_C=1$ as \emph{occupied}.

\begin{proof} [Proof of Lemma \ref{lem:cells}] 
	Let $\mathcal{F}$ be the collection of the $(d-1)$-dimensional closed faces of the cubes in $\mathcal{C}$.  For example, $[0,\ell]^{d-1}\times \{0\} \in \mathcal{F}$. 
	A set $\Gamma \subset \mathcal{F}$ is a \emph{contour} if 
	$\R^d \setminus \cup_{F\in \Gamma} F$ splits into exactly two connected components, one finite (\emph{inside}) and one infinite (\emph{outside}).  Let $\Int \Gamma\subset \mathcal{C}$ be the collection of cubes that are inside $\Gamma$, and $\partial_\mathrm{int} \Gamma \subset \Int \Gamma$ the cubes in the interior of $\Gamma$ that touch the contour (sharing a face, an edge or a corner), i.e., 
\begin{equation*} 
	\partial_\mathrm{int}\Gamma = \{ C \in \Int \Gamma \mid \bar C \cap \bigl(\bigcup_{F \in \Gamma} F\bigr) \neq \emptyset \}. 
\end{equation*}	
 Note that $\partial_\mathrm{int} \Gamma$ is connected and $|\partial_\mathrm{int} \Gamma|\geq c |\Gamma|$, for a suitable $\Gamma$-independent constant $c$. 
  
  Fix $C_0 \in \mathcal{C}$. The number of contours $\Gamma$ such that $C_0 \in\Int \Gamma$ and $|\Gamma|=n$ can be bounded by $c_1 n^{d/(d-1)} \exp( c_2 n)$, for some $d$-dependent constants $c_1,c_2>0$, see \cite[Section 5.3]{ruelle-book}. It follows that if $\alpha$ is sufficiently large, 
  \begin{equation*} 
  		\sum_{n=0}^\infty P\Bigl(\exists \Gamma: \ |\Gamma|=n,\ C_0\in \Int \Gamma,\ |\partial_\mathrm{int} \Gamma|\ \text{is empty} \Bigr) <\infty. 
  \end{equation*}
  The Borel-Cantelli lemma shows that $P$-almost surely, $C_0$ is enclosed in only finitely many contours with empty boundary $\partial_\mathrm{int} \Gamma$. Thus we can pick a cell adjacent from outside to the union of all such contours. This cell cannot be empty, and it cannot be surrounded by another contour with empty boundary. It follows that it must be a member of an infinite connected set of occupied cells. 
\end{proof}

\begin{proof}[Proof of (2) in Theorem \ref{thm:perco-gc}] 
	Pechersky and Yambartsev \cite{pechersky-yambartsev09} proved a similar statement in dimension $d=2$. We show that their proof can be adapted to $d\geq 2$. 
	Let $(r_0,r_1)$ be as in Assumption \ref{ass:basic}, $-m> \inf_{(r_0,r_1)} v =:-M$ and 
	$\tilde r_m\in (r_0,r_1)$ such that $v(\tilde r_m) \leq -m$. Since $v$ is continuous in $(r_0,r_1)$, we can find $\ell, \eps, \delta >0$  such that $\tilde r_m = \ell -\eps$,   $0<2 \eps <\ell$, and $v(r) \leq -m+\delta$ for all $r \in [\ell - 2\eps, \ell + 2 \eps]$. Because of
$\tilde r_m \geq r_0$, we also know that $v(r) \leq 0$ for all $r \geq \ell - \eps$. 

	We partition $\R^d$ into the cubes $C(\vect{k})$ of side-length $\ell$, as described before Lemma \ref{lem:cells}.  Let $\mathcal{R} \subset \mathcal{C}$ be a finite connected collection of cubes. Set $n:= |\mathcal{R}|$.  Consider the events
	 $\Omega^0(\mathcal{R})$ that there is no particle in $\Lambda_\mathcal{R}:=\cup_{C\in \mathcal{R}} C$ and 
	 $\Omega^1(\mathcal{R})$ that every cube $C \in \mathcal{R}$ contains exactly one particle, and this particle has distance $<\eps$ to the center of the cube.

	 Choose $L \in \ell \N$ large enough so that $\Lambda_\mathcal{R} \subset [-L,L]^d =:\Lambda$ and let $P_\Lambda = P_{\beta,\mu,\Lambda \mid \emptyset}$ 
	be the grand-canonical $(\beta,\mu)$-Gibbs measure in $\Lambda$ with free boundary conditions; write $\Xi_\Lambda = \Xi_{\Lambda\mid \emptyset}(\beta,\mu)$ for the associated partition function. Fix an arbitrary numbering $C_1,\ldots,C_n$ of the cells of $\mathcal{R}$ and write $\vect{u}^m$ for the center of $C_m$.  
%	We define 
%	 \begin{equation*} 
%	 		N(\mathcal{R}):= \sum_{C\in \mathcal{R}} |\{ C' \in \mathcal{R}\mid C'\ \text{is nearest neighbor of }C\}|. 
%	 \end{equation*} 
	 Note that if $\mathcal{R}$ is connected, each cube has at least one neighbor. 
%	  hence $N(\mathcal{R}) \geq |\mathcal{R}|=n$. 	 
%	 The following holds: 
	We have 
	 \begin{align*} 
	 	P_\Lambda(\Omega^1(\mathcal{R})) 
%	 		\frac{1}{\Xi_\Lambda} \sum_{k=0}^\infty \frac{z^{n+k}}{(n+k)!} 
%	 				\int_{\Lambda^{n+k}} e^{-\beta U(\vect{x})} \mathbf{1}_{\Omega^1(\mathcal{R})} 
%	 						(\{x_1,\ldots,x_{n+k}\}) \dd x_1\cdots \dd x_{n+k} \\
	 						& = \frac{1}{\Xi_\Lambda} \sum_{k=0}^\infty \frac{z^{n+k}}{k!}  
	 							\int_{B(\vect{u}^1, \eps)} \dd x_1 \cdots 	
	 							\int_{B(\vect{u}^n, \eps)} \dd x_n  \\
	 						& \qquad \qquad \times 
	 							\int_{(\Lambda \setminus \Lambda_\mathcal{R})^k} \dd y_1\cdots \dd y_k \, 
	 								e^{-\beta U(x_1,\ldots,x_n,y_1,\ldots,y_k) }.
%	 						& \geq z^n |B(0,\eps)|^n e^{\beta m N(\mathcal{R})/2} P_\Lambda( \Omega^0(\mathcal{R})). 
	\end{align*} 
	Because of our choice of $\ell$ and $\eps$, omitting the interaction between $y_k$'s and $x_i$'s in the integrand can only decrease the Boltzmann weight. Since in addition for every $\vect{x} \in \times_1^n B(u^i,\eps)$
	\begin{equation*}
		U(x_1,\ldots,x_n) \leq (-m+\delta) (n-1), 
	\end{equation*} 
	we obtain 
	\begin{equation*} 
		P_\Lambda\bigl(\Omega^1(\mathcal{R}) \bigr) \geq z^n |B(0,\eps)|^n e^{\beta (m-\delta) (n-1)  } P_\Lambda( \Omega^0(\mathcal{R})).
	\end{equation*}
	The same argument applies to finite volume Gibbs measures with tempered boundary conditions $\zeta$. Eq. \eqref{eq:dlr}  then shows that the previous inequality holds for every $P \in \mathcal{G}(\beta,\mu)$. It follows that, for every $P\in \mathcal{G}(\beta,\mu)$, 
	 \begin{equation*}
	 		P(\Omega^0(\mathcal{R}))  
	 		%\leq  \frac{P(\Omega^0(\mathcal{R}))}{P(\Omega^1(\mathcal{R}))}  
	 			 \leq \exp \Bigl[ - \beta |\mathcal{R}|\Bigl( \mu + \bigl(1- \frac{1}{|\mathcal{R}|} \bigr) (m-\delta) - \beta^{-1} \log |B(0,\eps)| \Bigr) \Bigr]. 
	 	\end{equation*}
	Lemma \ref{lem:cells} applies and shows that if 
	\begin{equation}  \label{eq:mulb}	
		\mu > - m +\delta - \beta ^{-1} \log  |B(0,\eps)| + \beta^{-1} \alpha_d
	\end{equation}
	then, $P$-almost surely, there is an infinite connected set of occupied cubes. Since the maximum distance between  particles in two adjacent cubes is $\sqrt{d+3}\, \ell$, it follows that the particle configuration $\omega$ has $P$-almost surely an infinite $R$-cluster, for every $R\geq \sqrt{d+3}\, \ell = \sqrt{d+3}(\tilde r_m + \eps)$.  	
	
	To conclude, fix $R> \sqrt{d+3}\, \tilde r_m$. Choose $\delta, \eps, \ell$ as above, with the additional requirement that $R> \sqrt{d+3} \ell$. Let $\beta(\delta,\eps)>(\alpha_d - \log |B(0,\eps)|)/\delta$. Then, for every $\beta \geq \beta(\delta,\eps)$, every $\mu > -m + 2\delta$, and every $P \in \mathcal{G}(\beta,\mu)$, there is an infinite cluster, $P$-almost surely. It follows that for sufficiently large $\beta$, $\mu_+(\beta;R_m) \leq - m +2 \delta$. 
Since $\delta$ can be chosen arbitrarily small, the proof is complete. 	
	\end{proof}	 

%Before we turn to the proof of non-percolation statements in Theorem \ref{thm:perco-gc}, we need an elementary probability lemma. 
%
%\begin{lemma} \label{lem:prob-bound} 
%	Let $P \in \mathcal{G}(\beta,\mu)$ for some $\beta >0$, $\mu \in \R$. For $\eps>0$, let $\Lambda_\eps := [-\eps/2,\eps/2]^d$. Then as $\eps \to 0$, 
%	\begin{equation*}
%		P(N_{\Lambda_\eps}(\omega) \geq 1) = \int_\Omega N_{\Lambda_\eps} (\omega) P(\dd \omega) + O(\eps^2).
%	\end{equation*}	
%\end{lemma} 
%
%\begin{proof} 
%	By Ruelle's superstability bounds \cite{ruelle70}, there is a $\xi>0$ such that the factorial moments of 
%		$N_{\Lambda_\eps}$ can be bounded as 
%	\begin{equation*} 
%		E\Bigl[ N_{\Lambda_\eps} ( N_{\Lambda_\eps} -1) \cdots N_{\Lambda_\eps}-n+1) \Bigr] 
%			\leq (\xi \eps)^n
%		\end{equation*}
%	for all $n \in \N$. Here $E$ is the expectation with respect to $P$. 
%	It follows that 
%	\begin{equation*} 
%		E [N_{\Lambda_\eps}] - P( N_{\Lambda_\eps}\geq 1) = E \Bigl[ (N_{\Lambda_\eps}-1) \mathbf{1}(N_{\Lambda_\eps}\geq 2)\Bigr] \leq E \Bigl[ N_{\Lambda_\eps}(N_{\Lambda_\eps}-1) \Bigr] \leq (\xi \eps)^2. \qedhere
%	\end{equation*}
%\end{proof}	 
	 
\begin{proof}[Proof of (1) in Theorem \ref{thm:perco-gc}]	 
		For $\beta>0$ and $k \in \N$, define the cluster partition functions 
		\begin{equation*} 
			Z_k^\cl(\beta):= \frac{1}{k!} \int_{(\R^d)^{k-1}} e^{-\beta U(0,x_2,\ldots,x_k)} 
					\mathbf{1}( \{0,x_2,\ldots,x_k\} \text{ is connected} \bigr) \dd x_2 \ldots \dd x_k 
		\end{equation*} 
		and let $R^\cl(\beta)$ be the radius of convergence of $\sum_{k=1}^\infty z^k Z_k^\cl(\beta)$, compare \cite[Prop. 3.3]{muermann75}).
%		\begin{equation*} 
%			f_\infty^\cl(\beta):=-  \lim_{k\to \infty} \frac{1}{\beta k} \log Z_k^\cl(\beta). 
%		\end{equation*} 
%		It is known that the limit exists and satisfies $f_\infty^\cl(\beta) \geq e_\infty - C \beta^{-1} \log \beta$ [ref]. Note that $\exp(\beta f_\infty^\cl(\beta))$ is the radius of convergence of $\sum_k z^k Z_k^\cl(\beta)$. 
		Let $P \in \mathcal{G}(\beta,\mu)$ be a Gibbs measure that can be obtained as a limit of finite volume, grand canonical Gibbs measures with empty boundary conditions. 
		Suppose that $z < R^\cl(\beta)$. Then $P(\text{there is an infinite cluster}) =0$
	\cite[Theorem 3.1]{muermann75}. M{\"u}rmann's proof moreover yields, for every cube $\Lambda \subset \R^d$ and all $k \in \N$, the bound 
		\begin{align*} 
				& P\Bigl( \exists x \in \omega \cap\Lambda:  |\mathcal{C}_\omega(x)| =k \Bigr) \\
				& \quad \leq  \frac{z^k}{k!} \int_{(\R^d)^k} e^{-\beta U(x_1,\ldots,x_k)} \mathbf{1}\bigl( \exists j \in \{1,\ldots,k\}:\ x_j \in \Lambda\bigr) \dd x_1 \ldots \dd x_k \\
				& \quad \leq k |\Lambda| z^k Z_k^\cl(\beta). 
		\end{align*} 
		Let $\xi>0$ be as a in the proof of Lemma \ref{lem:bc}. 
		If $P$ is shift-invariant, Eq. \eqref{eq:gpalm}  and Ruelle's superstability bounds \cite{ruelle70} show that  
		\begin{align*}
			|\Lambda| P^\circ( |\mathcal{C}_\omega(0)| =k) & = 
				\int_\Omega \Bigl|\{ x \in \omega \cap\Lambda:  |\mathcal{C}_\omega(x)| =k \}| \Bigr) P(\dd \omega) \\
				& = P\Bigl( \exists x \in \omega \cap\Lambda:  |\mathcal{C}_\omega(x)| =k \Bigr) 
					+O( (\xi |\Lambda|)^2 )
		\end{align*}		
		as $|\Lambda|\to 0$. It follows that 
		\begin{equation} \label{eq:rkbound} 
				\rho_k(P) \leq z^k Z_k^\cl(\beta). 
		\end{equation} 
		In order to go from shift-invariant limits of finite volume Gibbs measures to general Gibbs measures, we use the theory of Mayer expansions. It is well-known that for every $\beta>0$, some strictly positive $R^\may(\beta)>0$, and activities $z = \exp(\beta \mu) <R^\may(\beta)$, there is a unique Gibbs measure $P_{\beta,\mu} \in \mathcal{G}(\beta,\mu)$; furthermore, the pressure and the correlation functions admit absolutely convergent expansions in powers of $z$ (with temperature-dependent coefficients)  \cite[Theorem 5.7]{ruelle70}. The measure $P_{\beta,\mu}$ is shift-invariant. 		
		Since every finite volume Gibbs measure converges to an infinite volume Gibbs measure, this limit must be $P_{\beta,\mu}$, and we can apply the previous considerations to $P_{\beta,\mu}$. Thus we see that when $z < \min(R^\cl(\beta), R^\may(\beta))$, there is a unique Gibbs measure. It is shift-invariant, has no infinite cluster ($P$-almost surely), and satisfies Eq. \eqref{eq:rkbound}. 
		
	Next,
		recall 
		\begin{equation*} 
				Z_k^\cl(\beta) \leq e^{- \beta k e_\infty} e^{k} |B(0,R)|^{k-1},\quad 
					R^\may(\beta) \geq \frac{\exp(\beta e_\infty)}{\beta |||v|||}
		\end{equation*} 
		where $|||v|||:=|B(0,r_\mathrm{hc})| + \int_{|x|>r_\mathrm{hc}} |v(|x|)| \dd x$, see 
		\cite[Proof of Lemma 4.1]{jkm11} or \cite[Proof of Prop. 3.1]{muermann75} for the first inequality and \cite[Theorem 2.1]{poghosyan-ueltschi09} for the second inequality. We deduce
		\begin{equation*}
				\liminf_{\beta \to \infty} \beta^{-1} \log \min\bigl(R^\cl(\beta),R^\may(\beta)\bigr)
					\geq e_\infty. 
		 \end{equation*} 
		 Since every $\mu_-<\beta^{-1} \log \min(R^\cl(\beta), R^\may(\beta))$ satisfies condition \eqref{eq:cond-mu-minus}, we obtain the desired lower bound on $\mu_-(\beta;R)$. The upper bound \eqref{eq:decay-gc} of $\rho_k(P)$ follows from Eq. \eqref{eq:rkbound} and the upper bound of $Z_k^\cl(\beta)$. 
\end{proof}

\subsection{Canonical ensemble} 

Theorem \ref{thm:perco-can} is deduced from Theorem \ref{thm:perco-gc} with the help of a good control of the density as a function of the chemical potential. Recall that the pressure $p(\beta,\mu) = \sup_{\rho}(\rho \mu - f(\beta,\rho))$ is a convex function of $\mu$. Therefore the derivative $\rho(\beta,\mu):=\partial_\mu p(\beta,\mu)$ exists almost everywhere and is an increasing function of $\mu$. Moreover, if $\rho = \rho(\beta,\mu)$, then $\mathcal{G}_\theta(\beta,\rho)\subset \mathcal{G}(\beta,\mu)$.
If $p(\beta,\mu)$ has different left and right derivatives $\rho_\mathrm{l}$ and $\rho_\mathrm{r}$ with respect to $\mu$,  the previous inclusion holds for every $\rho \in [\rho_\mathrm{l},\rho_\mathrm{r}]$. 

\begin{proof}[Proof of (1) in Theorem \ref{thm:perco-can}] 
	Remember $R^\may(\beta)$ and $R^\cl(\beta)$ from the proof of (1) in Theorem \ref{thm:perco-gc}. 
	For $\exp(\beta\mu)<R^\may(\beta)$, the pressure is differentiable in $\mu$ and $\rho(\beta,\mu)$ is well-defined. 
	Fix $\eps>0$ and let $\beta_\eps >0$ large enough so that for $\beta \geq \beta_\eps$, 
	$\exp( \beta [e_\infty - \eps]) < \min (R^\may(\beta), R^\cl(\beta))$. 
	For $\beta \geq \beta_\eps$ and $\rho <\rho(\beta,e_\infty - \eps)$, we  know that $\rho = \rho(\beta,\mu)$ for a \emph{unique} $\mu$; for this $\mu$, $\mathcal{G}_\theta(\beta,\rho) = \mathcal{G}(\beta,\mu) = \{P\}$ with a unique Gibbs measure $P$. From the proof of Theorem \ref{thm:perco-gc}, we see that $P$ assigns probability zero to the event that there is an infinite cluster. As a consequence, $\rho_-(\beta;R) \geq \rho(\beta,e_\infty - \eps)$. 
	Since 
	\begin{equation} \label{eq:rho-low-temp} 
		\lim_{\beta \to \infty}\beta^{-1} \log \rho(\beta,e_\infty - \eps) 
				= - \inf_{k \in \N} \Bigl(E_k - k (e_\infty - \eps) \Bigr)
	\end{equation}	
	and $\inf_{\eps>0} \inf_{k \in \N} [E_k - k (e_\infty - \eps)] = \nu^*$
	\cite{jansen12}, we deduce $\rho_-(\beta;R) \geq \exp(- \beta \nu^*(1+o(1)))$. 
	
	In addition, if $\nu >\nu^*$ and $\beta$ is sufficiently large, we see that $\rho = \rho(\beta,\mu)$ for a unique $\mu < \beta^{-1} \log \min (R^\may(\beta),R^\cl(\beta))$, 
	and the remaining part of Theorem \ref{thm:perco-can} follows from the corresponding statement in Theorem \ref{thm:perco-gc}.  
	The claim on the uniformity of the constants stated after Theorem \ref{thm:perco-can}  follows by combining the inequalities \eqref{eq:decay-gc}, \eqref{eq:rkbound}, and \eqref{eq:rho-low-temp}. 
\end{proof}

\begin{proof}[Proof of (2) in Theorem \ref{thm:perco-can}]
	Fix $\eps>0$. By Theorem \ref{thm:perco-gc}, there is a $\beta_\eps>0$ such that for every $\beta \geq \beta_\eps$, $\mu_+(\beta;R_m) \leq - m + \eps$. Let $\rho(\beta,-m+\eps)$ be the left derivative of $p(\beta,\mu)$ with respect to $\mu$ at $\mu = -m/2+\eps$. If $\rho \geq \rho_\eps(\beta)$, then $\mathcal{G}_\theta(\beta,\rho) \subset \mathcal{G}(\beta,\mu)$ for some $\mu \geq - m +\eps$. When $\beta \geq \beta_\eps$, it follows that this $\mu$ is larger than $\mu_+(\beta;R_m)$, and Theorem \ref{thm:perco-gc} tells us that for every $P \in \mathcal{G}_\theta(\beta,\rho) \subset \mathcal{G}(\beta,\mu)$, there is an infinite cluster, $P$-almost surely. Thus $\rho_+(\beta;R_m) \leq  \rho(\beta,-m+\eps)$ and $\liminf_{\beta \to \infty} \rho_+(\beta,\mu) \leq \rho_m$. 
\end{proof}

\begin{proof}[Proof of Prop. \ref{prop:soft-trans}] 
	Prop. \ref{prop:soft-trans} is a consequence of Theorem 1.8 in \cite{jkm11} and our Theorem \ref{thm:equiv}: the result from \cite{jkm11} says that the bound of our proposition holds for every minimizer $(\rho_k)_{k\in \N}$ of $f(\beta,\rho, \cdot)$, and Theorem \ref{thm:equiv} says that for every $P \in \mathcal{G}_\theta(\beta,\rho)$, the vector $(\rho_k(P))_{k \in \N}$ is a minimizer of $f(\beta,\rho,\cdot)$.  
\end{proof}

\appendix

\section{Lattice gas and Ising model} \label{app:lattice}  

Here we explain how our results relate to known results for lattice systems. 
We shall be very sketchy and refer the reader to to the review \cite{ghm} and the references therein for precise definitions and proofs.

Consider the nearest neighbor lattice gas in $\Z^d$  ($d\geq 2$)
with hard-core on-site interaction and attractive nearest-neighbor interaction. As is well-known, this model can be recast as an Ising model. Occupied lattice sites ($n_x=1$) are mapped to spin $\sigma_x=+1$ and empty lattice sites ($n_x=0$) to spin $\sigma_x = -1$; thus $\sigma_x = 2 n_x-1$. Formally, we have 
\begin{equation*}
	- \frac{J}{2}\sum_{|x-y|=1} n_xn_y - \mu \sum_{x} n_x 
		= - \frac{J}{8} \sum_{|x-y|=1} \sigma_x \sigma_y - \frac{dJ+ \mu}{2} \sum_x \sigma_x + \const, 
\end{equation*} 
which determines the external magnetic field $h$ of the Ising model in terms of the chemical potential $\mu$ of the lattice gas as $h=(\mu + dJ)/2$.  The parameter $J>0$ measures both the strength of attraction between particles and the strength of the ferromagnetic coupling between spins. 

The relevant notion of percolation is \emph{dependent site percolation}. \emph{Dependent} refers to the underlying measure on $\{0,1\}^{\Z^d}$, which is a Gibbs measure rather than a product of Bernoulli measures, and \emph{site percolation} refers to the notion of connectivity -- two occupied lattice sites $x,y\in \Z^d$ are connected if there is a path $(x_1,\ldots,x_n)$ of nearest neighbor ($|x_{j+1}-x_j|=1$) joining $x_1=x$ and $x_n=y$, such that each site of the path is occupied, $n_{x_j}=1$. In the Ising picture, we are interested in \emph{percolation of $+$clusters}. 

Our first remark is that at low temperature, the phase transition and the percolation transition coincide; this observation, as alluded to in the introduction, is the driving motivation for the present article's investigation.  Fix a temperature $T>0$ and vary the external field $h$ (or the chemical potential $\mu$). At temperatures above the \emph{Curie temperature} $T_\mathrm{C}>0$, there is no phase transition (the pressure stays analytic and the Gibbs measure is unique); at $T<T_\mathrm{C}$, there is a first-order phase transition at $h=0$. On the other hand, let $P_{T,h}^+$ be the Gibbs measure with $+$ boundary conditions. We know that for all $T>0$, there is a threshold $h(T)\in \R$ such that the $P_{T,h}^+$-probability of having an infinite cluster is $0$ at external fields $h<h(T)$, and $1$ at fields $h>h(T)$. Furthermore, there is a temperature $T_+>0$ such that $h(T) =0$ for $T<T_+$ and $h(T) <0$ for $T>T_+$ \cite{aizenman-bricmont-lebowitz87}, and it is known that $T_+<T_\mathrm{C}$ in high dimensions. Thus at temperatures above the Curie temperature, there is a percolation transition but no phase transition; at temperatures between $T_+$ and $T_\mathrm{C}$, as $h$ is increased, the percolation transition happens \emph{before} the phase transition; and below $T_+$, the percolation transition coincides with the phase transition. 

The next observation is that the value at which the transition takes place is consistent with our grand-canonical Theorem \ref{thm:perco-gc}: indeed, the zero external field $h=0$ corresponds, in the lattice gas picture, to a chemical potential $\mu = - d J$, which can be interpreted as a ground state energy per particle -- if in $\Z^d$ every lattice site is occupied, the energy per particle is $e_\infty = - dJ$. 

Finally, this agreement of thresholds extends to the canonical ensemble. At $T<T_\mathrm{C}$ and external field $h=0$, there are two shift-invariant Gibbs measures $P^\pm_{T,h}$, with magnetizations $m_{\pm}(T)$. Low temperature contour expansions show that as $T\to 0$, 
$m_\pm (T) = \pm 1 + O( \exp( - J/T))$. Transforming as $\rho_\pm(T) = (m_\pm(T)+1)/2$,  we obtain that the corresponding curves for the lattice gas satisfy $\rho_-(T) = O(\exp( - J/T))$ and $\rho_+(T) = 1 + O(\exp(-J/T))$, which compares nicely with Theorem \ref{thm:perco-can}: for $T<T_+$, the magnetizations $m_\pm(T)$ of the Ising model are also percolation thresholds (this is not true for $T_+<T<T_\mathrm{C}$).

\subsection*{Acknowledgments}
I gratefully acknowledge financial support by the
DFG-Forscher\-gruppe~FOR718 ``Analysis and stochastics in complex physical systems'' and the Hausdorff Research Institute for Mathematics (Bonn, Germany). I am indebted to Prof. A. van Enter for helpful remarks on an earlier version of the manuscript. 

%\bibliographystyle{amsalpha}
%\bibliography{dilute}

\begin{thebibliography}{GHM01}

\bibitem[ABL87]{aizenman-bricmont-lebowitz87}
M. Aizenman, J. Bricmont and J. L. Lebowitz,  
	\emph{Percolation of the minority spins in high-dimensional {I}sing models}, 
	J. Statist. Phys.\textbf{49} (1987), 
	859--865.

\bibitem[Ar12]{aristoff12}
D.~Aristoff, \emph{Percolation for a point process with hard core},
  arxiv:1207.2136v1 [math-ph], 2012.

\bibitem[B04]{blanc04}
X.~Blanc, \emph{{Lower bound for the interatomic distance in Lennard-Jones
  clusters}}, Comput. Optim. Appl. \textbf{29} (2004), 5–--12.

\bibitem[BF78]{brydges-federbush78}
D.~Brydges and P.~Federbush, \emph{A new form of the Mayer expansion in
  classical statistical mechanics}, J. Math. Phys. \textbf{19} (1978),
  2064–--2067.

\bibitem[D70]{dobrushin70}
R.~L. Dobrushin, \emph{Gibbsian random fields for particles without hard core},
  Teoret. Mat. Fiz. \textbf{4} (1970), 101--118; English translation in Theoret. and Math. Phys. \textbf{4} (1970), 705–-719.

\bibitem[DVJ08]{daley-vere-jones-vol2}
D.~J. Daley and D.~Vere-Jones, \emph{An introduction to the theory of point
  processes. {V}ol. {II}}, second ed., Probability and its Applications (New
  York), Springer, New York, 2008.

\bibitem[DZ98]{dembo-zeitouni-book}
A.~Dembo and O.~Zeitouni, \emph{Large deviations techniques and applications},
  second ed., Applications of Mathematics (New York), vol.~38, Springer-Verlag,
  New York, 1998.

\bibitem[G94]{georgii94}
H.-O. Georgii, \emph{Large deviations and the equivalence of ensembles for
  {G}ibbsian particle systems with superstable interaction}, Probab. Theory
  Relat. Fields \textbf{99} (1994), 171–--195.

\bibitem[G95]{georgii95}
\bysame, \emph{The equivalence of ensembles for classical particle systems}, J.
  Statist. Phys. \textbf{80} (1995), 1341--1378.

\bibitem[GHM01]{ghm}
H.-O. Georgii, O.~H{\"a}ggstr{\"o}m, and C.~Maes, \emph{The random geometry of
  equilibrium phases}, Phase transitions and critical phenomena, {V}ol. 18,
  Phase Transit. Crit. Phenom., vol.~18, Academic Press, San Diego, CA, 2001,
  pp.~1--142.

\bibitem[GZ93]{georgii-zessin93}
H.-O. Georgii and H.~Zessin, \emph{Large deviations and the maximum entropy
  principle for marked point random fields}, Probab. Theory Related Fields
  \textbf{96} (1993), 177–--204.

\bibitem[J12]{jansen12}
S.~Jansen, \emph{Mayer and virial series at low temperature}, J. Stat. Phys.
  \textbf{147} (2012), 678--706.

\bibitem[JKM11]{jkm11}
S.~Jansen, W.~K{\"o}nig, and B.~Metzger, \emph{Large deviations for cluster
  size distributions in a continuous classical many-body system},
  arXiv:1107.3670v2 [math.PR], 2011.

\bibitem[MR96]{meester-roy-book}
R.~Meester and R.~Roy, \emph{Continuum percolation}, Cambridge Tracts in
  Mathematics, vol. 119, Cambridge University Press, Cambridge, 1996.

\bibitem[M75]{muermann75}
M.~G. M{\"u}rmann, \emph{Equilibrium distributions of physical clusters}, Comm.
  Math. Phys. \textbf{45} (1975), 233--246.

\bibitem[PU09]{poghosyan-ueltschi09}
S.~Poghosyan and D.~Ueltschi, \emph{Abstract cluster expansion with
  applications to statistical mechanical systems}, J. Math. Phys. \textbf{50}
  (2009), 053509, 17.

\bibitem[PY09]{pechersky-yambartsev09}
E.~Pechersky and A.~Yambartsev, \emph{Percolation properties of the non-ideal
  gas}, J. Stat. Phys. \textbf{137} (2009), no.~3, 501--520.

\bibitem[R69]{ruelle-book}
D.~Ruelle, \emph{Statistical mechanics: {R}igorous results}, W. A. Benjamin,
  Inc., New York-Amsterdam, 1969.

\bibitem[R70]{ruelle70}
\bysame, \emph{Superstable interactions in classical statistical mechanics},
  Comm. Math. Phys. \textbf{18} (1970), 127--159.

\bibitem[T06]{theil06}
F.~Theil, \emph{A proof of crystallization in two dimensions}, Comm. Math.
  Phys. \textbf{262} (2006), 209--236.

\bibitem[Z08]{zessin08}
H.~Zessin, \emph{A theorem of {M}ichael {M}{\"u}rmann revisited}, Izv. Nats.
  Akad. Nauk Armenii Mat. \textbf{43} (2008), 69--80, translation in J.
  Contemp. Math. Anal. \textbf{43} (2008),  50–-58.
\end{thebibliography}

\providecommand{\bysame}{\leavevmode\hbox to3em{\hrulefill}\thinspace}
\providecommand{\MR}{\relax\ifhmode\unskip\space\fi MR }
\providecommand{\MRhref}[2]{%
  \href{http://www.ams.org/mathscinet-getitem?mr=#1}{#2}
}
\providecommand{\href}[2]{#2}

\end{document}